\documentclass[12pt]{article}%
\usepackage{graphicx}
\usepackage[intlimits]{amsmath}
\usepackage{latexsym}
\usepackage{amsfonts}
\usepackage{amssymb}%
\makeatletter
\newfont{\footsc}{cmcsc10 at 8truept}
\newfont{\footbf}{cmbx10 at 8truept}
\newfont{\footrm}{cmr10 at 10truept}
\newtheorem{theorem}{Theorem}[section]

\newtheorem{conjecture}[theorem]{Conjecture}
\newtheorem{corollary}[theorem]{Corollary}

\newtheorem{problem}[theorem]{Problem}
\newtheorem{proposition}[theorem]{Proposition}
\newtheorem{question}[theorem]{Question}

\newenvironment{proof}[1][Proof]{\noindent{\textbf {#1}  }}  {\hfill$\Box$\bigskip}

\begin{document}

\title{\textbf{Extrema of graph eigenvalues}}
\author{Vladimir Nikiforov\thanks{Department of Mathematical Sciences, University of
Memphis, Memphis TN 38152, USA; \textit{email: vnikifrv@memphis.edu}}}
\maketitle

\begin{abstract}
In 1993 Hong asked what are the best bounds on the $k$'th largest eigenvalue
$\lambda_{k}(G)$ of a graph $G$ of order $n$. This challenging question has
never been tackled for any $2<k<n$. In the present paper tight bounds are
obtained for all $k>2,$ and even tighter bounds are obtained for the $k$'th
largest singular value $\lambda_{k}^{\ast}(G).$

Some of these bounds are based on Taylor's strongly regular graphs, and other
on a method of Kharaghani for constructing Hadamard matrices. The same kind of
constructions are applied to other open problems, like Nordhaus-Gaddum
problems of the kind: \emph{How large can }$\lambda_{k}(G)+\lambda
_{k}(\overline{G})$ \emph{be}$?$

These constructions are successful also in another open question: \emph{How
large can the Ky Fan norm} $\lambda_{1}^{\ast}(G)+\cdots+\lambda_{k}^{\ast
}(G)$\emph{ be}$?$ Ky Fan norms of graphs generalize the concept of graph
energy, so this question generalizes the problem for maximum energy graphs.

In the final section, several results and problems are restated for
$(-1,1)$-matrices, which seem to provide a more natural ground for such
research than graphs.

Many of the results in the paper are paired with open questions and problems
for further study.\smallskip

\smallskip


\end{abstract}

\section{Introduction}

\textit{What are the best possible lower and upper bounds on the }%
$k$\textit{'th largest eigenvalue }$\lambda_{k}\left(  G\right)  $\textit{ of
a graph }$G$\textit{ of order }$n?$

Yuan Hong raised this fundamental question in 1993, in his paper \cite{Hon93}.
Apparently he was unaware that five years earlier Powers \cite{Pow89}, p. 5,
had published the following result:\medskip

\emph{If }$G$\emph{ is a connected graph of order }$n$\emph{, then }%
\begin{equation}
\lambda_{k}\left(  G\right)  \leq\left\lfloor n/k\right\rfloor . \label{Powbo}%
\end{equation}
It is not hard to realize that if inequality (\ref{Powbo}) were true, it would
essentially answer Hong's question. Alas, it is not. Its proof is flawed, and
it fails for all $k\geq5.$ However, in fairness, the inequality certainly
holds for $k=1,2,$ while for $k=3,4,$ it is a challenging open problem. For
that matter, except for $k=2,$ connectedness is an irrelevant premise in these questions.

Other than this unsuccessful attempt, the general problem of Hong has never
been tackled seriously. This is all the more inexplicable, as the problem is
indeed challenging, easier for some values of $k,$ and well beyond reach for
other. What is more, Hong's problem is not a backyard puzzle that is of
interest only to spectral graph theorists; it is related to other fundamental
areas of combinatorics and analysis, like existence of symmetric Hadamard
matrices, Ramsey's theorem, and extremal norms of graphs. We feel that the
appeal and the importance of Hong's problem should attract the attention of
many a researcher, and to this effect we take a few steps in the present paper.

We shall extend Hong's problem to the $k$'th largest singular value
$\lambda_{k}^{\ast}(G)$ of $G,$ and shall give upper and lower bounds on
$\lambda_{k}\left(  G\right)  $ and $\lambda_{k}^{\ast}(G)$, including a few
exact results and asymptotics for the general cases. Many open problems and
questions will be raised to outline directions for further study.

Two fundamental results underpin our constructions: first, the strongly
regular graphs of Taylor \cite{Tay71,Tay77}; and second, Kharaghani's method
for constructing Hadamard matrices \cite{Kha85}. These two topics deserve to
be known better in spectral graph theory, as their potential uses seem indeed unlimited.

We shall apply the same constructions to other open problems, like, e.g.,
Nordhaus-Gaddum problems of the following kind:\medskip\ 

\emph{If }$G$\emph{ is a graph of order }$n$ \emph{and }$\overline{G}$\emph{
is its complement,} \emph{how large can }$\lambda_{k}\left(  G\right)
+\lambda_{k}(\overline{G})$\emph{ be}$?\medskip$\emph{ }

Such problems have been raised in \cite{Nik07}, and some recent progress has
been given in \cite{NiYu14}. We shall exhibit a new infinite family of
solutions, and derive general asymptotics.\medskip

We make also progress with an open problem about maximum Ky Fan norms of
graphs. Recall that the \emph{Ky Fan }$k$\emph{-norm} of a graph $G$ is
defined as $\lambda_{1}^{\ast}\left(  G\right)  +\cdots+\lambda_{k}^{\ast
}\left(  G\right)  $. In \cite{Ni11}, the following problem has been
raised:\medskip

\emph{If }$G$\emph{ is a graph of order }$n,$\emph{ how large can the Ky Fan
}$k$\emph{-norm of }$G$\emph{ be? \medskip}

Note that the Ky Fan $n$-norm is also known as the \emph{trace norm} of $G,$
and has been extensively studied under the name \emph{graph energy,} a concept
introduced by Gutman in \cite{Gut78}. There is vast research on graph energy,
but Ky Fan norms can open even larger horizons. Here we shall solve the above
problem whenever $k$ is an even square.\medskip

The structure of the paper is as follows: In Section \ref{Hp} we present
results on Hong's problem. Section \ref{NGs} is dedicated to spectral
Nordhaus-Gaddum problems, and in Section \ref{KFs} we present results on Ky
Fan norms of graphs. Section \ref{NBS} is for reader's convenience: it
contains references, notation and basics on Weyl's inequalities, blow-ups of
graphs, Taylor's strongly regular graphs, and symmetric Latin squares. Section
\ref{ps} contains the proofs of several theorems, which are either too
involved of would have disrupted the exposition. Finally, Section
\ref{RM}contains a selection of the presented results and problems, translated
from graphs to symmetric $\left(  -1,1\right)  $-matrices. It becomes obvious
that such matrices provide a more balanced and natural setup for such
research. In particular, the strongly regular graphs of Taylor are translated
into a $\left(  -1,1\right)  $-matrix with a rather peculiar spectrum.

\section{\label{Hp}Hong's problem and its variations}

Let $G$ be a graph of order $n.$ The eigenvalues $\lambda_{1}\left(  G\right)
,\ldots,\lambda_{n}\left(  G\right)  $ of $G$ are the eigenvalues of its
adjacency matrix $A\left(  G\right)  ,$ ordered as $\lambda_{1}\left(
G\right)  \geq\cdots\geq\lambda_{n}\left(  G\right)  .$ The singular values
$\lambda_{1}^{\ast}\left(  G\right)  ,\ldots,\lambda_{n}^{\ast}\left(
G\right)  $ of $G\ $are the absolute values of $\lambda_{1}\left(  G\right)
,\ldots,\lambda_{n}\left(  G\right)  ,$ ordered as $\lambda_{1}^{\ast}\left(
G\right)  \geq\cdots\geq\lambda_{n}^{\ast}\left(  G\right)  .$ In particular,
$\lambda_{1}^{\ast}\left(  G\right)  =\lambda_{1}\left(  G\right)  ,$ and
\[
\left\{  \lambda_{1}^{\ast}\left(  G\right)  ,\lambda_{2}^{\ast}\left(
G\right)  ,\ldots,\lambda_{n}^{\ast}\left(  G\right)  \right\}  =\left\{
\lambda_{1}\left(  G\right)  ,\left\vert \lambda_{2}\left(  G\right)
\right\vert ,\ldots,\left\vert \lambda_{n}\left(  G\right)  \right\vert
\right\}  .
\]
Note that, in general, graph singular values cannot be reduced to graph
eigenvalues, as the two multisets may be ordered very differently.

We shall extend the original problem of Hong to the largest singular values of
graphs, and shall bring to the fore the study of the smallest eigenvalues.
These changes correspond to the present day interest in these spectral
parameters.\medskip

Let $n\geq k\geq1.$ Define the functions $\lambda_{k}\left(  n\right)  ,$
$\lambda_{-k}\left(  n\right)  ,$ and $\lambda_{k}^{\ast}\left(  n\right)  $
as\ \
\begin{align*}
\lambda_{k}\left(  n\right)   &  =\max_{v\left(  G\right)  =n}\text{ }%
\lambda_{k}\left(  G\right)  ,\\
\lambda_{-k}\left(  n\right)   &  =\max_{v\left(  G\right)  =n}\text{
}\left\vert \lambda_{n-k+1}\left(  G\right)  \right\vert ,\\
\lambda_{k}^{\ast}\left(  n\right)   &  =\max_{v\left(  G\right)  =n}%
\text{\ }\lambda_{k}^{\ast}\left(  G\right)  .
\end{align*}
Note that if $n\geq\binom{2k-1}{k-1},$ then $\lambda_{n-k+1}\left(  G\right)
$ is always nonpositive (see Theorem \ref{lob} below), and so, $\min_{v\left(
G\right)  =n}$ $\lambda_{n-k+1}\left(  G\right)  =-\lambda_{-k}\left(
n\right)  ;$ thus, the use of the absolute value in the definition of
$\lambda_{-k}\left(  n\right)  $ is just to make the setup more
uniform.\medskip

Now, we restate Hong's problem into two separate problems:

\begin{problem}
\label{mpro}For any $k\geq1,$ find $\lambda_{k}\left(  n\right)  ,$
$\lambda_{-k}\left(  n\right)  ,$ and \ $\lambda_{k}^{\ast}\left(  n\right)
.$
\end{problem}

\begin{problem}
\label{minpro} For any $k\geq1,$ find $\min\limits_{v\left(  G\right)  =n}%
$\ $\lambda_{k}\left(  G\right)  ,$ $\max\limits_{v\left(  G\right)  =n}$
$\lambda_{n-k+1}\left(  G\right)  ,$ and $\min\limits_{v\left(  G\right)  =n}%
$\ $\lambda_{k}^{\ast}\left(  G\right)  .$
\end{problem}

This separation is justified, as the two problems are of incomparable
difficulty: indeed, presently the full solution of Problem \ref{mpro} is
beyond reach, while we shall dispose of Problem \ref{minpro} right away.

Indeed, if $n\geq k,$ the complete graph $K_{n}$ of order $n$ satisfies
$\lambda_{k}\left(  K_{n}\right)  =-1;$ and likewise, the edgeless graph
$\overline{K}_{n}$ of order $n$ satisfies $\lambda_{n-k+1}(\overline{K}%
_{n})=0$ and $\lambda_{k}^{\ast}(\overline{K}_{n})=0.$ These bounds are also
best possible: indeed, obviously $\lambda_{k}^{\ast}\left(  G\right)  \geq0$
for any graph $G$ of order $n\geq k$; and for $\lambda_{k}$ and\ $\lambda
_{n-k+1}$ this fact is true in view of the following theorem:

\begin{theorem}
\label{lob}If $n\geq\binom{2k-1}{k-1}$ and $G\ $is a graph of order $n$, then
\begin{equation}
\lambda_{k}\left(  G\right)  \geq-1\text{ \ \ and \ \ }\lambda_{n-k+1}\left(
G\right)  \leq0. \label{bo1}%
\end{equation}

\end{theorem}

\begin{proof}
We shall use Ramsey's theorem, whose application in graph spectra has been
pioneered only recently, in \cite{NiYu14} and \cite{ZhCh14}.

The classical bound of Erd\H{o}s and Szekeres implies that every graph of
order at least $\binom{2k-1}{k-1}$ contains either a complete graph on $k+1$
vertices or an independent set on $k$ vertices. If $G$ contains a complete
graph on $k+1$ vertices, then Cauchy's interlacing theorem implies that
\[
\lambda_{k}\left(  G\right)  \geq\lambda_{k}\left(  K_{k+1}\right)  =-1\text{
\ and \ }\lambda_{n-k+1}\left(  G\right)  \leq\lambda_{2}\left(
K_{k+1}\right)  =-1,
\]
so (\ref{bo1}) follows. If $G$ contains an independent set on $k$ vertices,
then Cauchy's interlacing theorem implies that
\[
\lambda_{k}\left(  G\right)  \geq\lambda_{k}(\overline{K}_{k})=0\text{ \ and
\ \ \ }\lambda_{n-k+1}\left(  G\right)  \leq\lambda_{1}(\overline{K}_{k})=0,
\]
and (\ref{bo1}) follows again.
\end{proof}

\medskip

Now, let us turn to Problem \ref{mpro}. We start with an observation, which
exhibits some dependencies between the functions $\lambda_{k}\left(  n\right)
,$ $\lambda_{-k}\left(  n\right)  ,$ and $\lambda_{k}^{\ast}\left(  n\right)
.$

\begin{proposition}
\label{spro}If $k\geq2,$ then
\[
\lambda_{k}\left(  n\right)  \leq\lambda_{k}^{\ast}\left(  n\right)  ,\text{
\ \ }\lambda_{-k+1}\left(  n\right)  \leq\lambda_{k}^{\ast}\left(  n\right)
,\text{ \ \ and \ \ \ }\lambda_{k}\left(  n\right)  +1\leq\lambda
_{-k+1}\left(  n\right)  .
\]

\end{proposition}

The first two inequalities follow from the definition of $\lambda_{k}^{\ast
}\left(  n\right)  .$ For the last inequality recall that Weyl's inequalities
(see \ref{WS}) imply that if $G$ is a graph of order $n$ and $2\leq k\leq n,$
then $\lambda_{k}\left(  G\right)  +\lambda_{n-k+2}(\overline{G})\leq-1,$ and
so, $\lambda_{k}\left(  G\right)  +1\leq|\lambda_{-k+1}(\overline{G}%
)|.$\medskip\ 

Our first goal is to establish concise asymptotics of $\lambda_{k}\left(
n\right)  ,$ $\lambda_{-k}\left(  n\right)  ,$ and $\lambda_{k}^{\ast}\left(
n\right)  $. To this effect, for any integer $k\geq1,$ define the real numbers
$c_{k},$ $c_{-k},$ and $c_{k}^{\ast}$ as
\begin{align*}
c_{k} &  =\sup\left\{  \text{ \ \ }\lambda_{k}\left(  G\right)
/n\text{\ \ \ \ \ \ }:\text{ }G\text{ is a graph of order }n\geq k\right\}
,\\
c_{-k} &  =\sup\left\{  \left\vert \lambda_{n-k+1}\left(  G\right)
\right\vert /n\text{ }:\text{ }G\text{ is a graph of order }n\geq k\right\}
,\\
c_{k}^{\ast} &  =\sup\left\{  \text{$\ $\ \ $\ \lambda_{k}^{\ast}\left(
G\right)  /n$ \ \ \ \ }:\text{ }G\text{ is a graph of order }n\geq k\right\}
.
\end{align*}
Clearly, these definitions imply that if $G$ is a graph of order $n,$ then%
\[
\lambda_{k}\left(  G\right)  \leq c_{k}n,\text{ \ \ }\lambda_{n-k+1}\left(
G\right)  \geq-c_{-k}n,\text{ \ \ \ and \ \ \ }\lambda_{k}^{\ast}\left(
G\right)  \leq c_{k}^{\ast}n.
\]
The above bounds are handy, and fortunately they are also tight, as shown by
the following theorem, which can be proved with the methods of \cite{Nik06}:

\begin{theorem}
\label{limth}For every $k\geq1,$%
\[
\lim_{n\rightarrow\infty}\text{ }\lambda_{k}\left(  n\right)  /n=c_{k},\text{
\ }\lim_{n\rightarrow\infty}\text{ }\lambda_{-k}\left(  n\right)
/n=c_{-k},\text{ \ \ and \ \ }\lim_{n\rightarrow\infty}\text{ }\lambda
_{k}^{\ast}\left(  n\right)  /n\text{ \ }=c_{k}^{\ast}.
\]

\end{theorem}

Therefore, a good deal of information about $\lambda_{k}\left(  G\right)
,$\ $\lambda_{n-k+1}\left(  G\right)  ,$\ and $\lambda_{k}^{\ast}\left(
G\right)  $ can be obtained if we knew the constants $c_{k},$ $c_{-k},$ and
$c_{k}^{\ast}$ or some good estimates thereof.

\subsection{\textbf{Upper bounds on }$\lambda_{k}\left(  n\right)  ,$
$\lambda_{-k}\left(  n\right)  ,$ and $\lambda_{k}^{\ast}\left(  n\right)  $}

Next, we give an easy upper bound on $\lambda_{k}^{\ast}\left(  n\right)  ;$
later, by much harder work, we shall show that this bound is almost as good as
one can get.

Note that if $G$ is a graph of order $n,$ with $e\left(  G\right)  $ edges and
adjacency matrix $A,$ then%
\[
\lambda_{1}^{2}\left(  G\right)  +\lambda_{2}^{\ast2}\left(  G\right)
+\cdots+\lambda_{n}^{\ast2}\left(  G\right)  =\mathrm{tr}\text{ }%
A^{2}=2e\left(  G\right)  .
\]
Hence, using the inequality $\lambda_{1}\left(  G\right)  \geq2e\left(
G\right)  /n$ and the AM-GM inequality, one finds that
\[
\lambda_{2}^{\ast2}\left(  G\right)  +\cdots+\lambda_{k}^{\ast2}\left(
G\right)  \leq2e\left(  G\right)  -\lambda_{1}^{2}\left(  G\right)
\leq2e\left(  G\right)  -\left(  \frac{2e\left(  G\right)  }{n}\right)
^{2}\leq\frac{n^{2}}{4}.
\]
Therefore, $\left(  k-1\right)  \lambda_{k}^{\ast2}\left(  G\right)  \leq
n^{2}/4,$ and Proposition \ref{spro} implies the following bounds:

\begin{theorem}
\label{th1}If $n\geq k\geq2$ and $G$ is a graph of order $n,$ then
\[
\lambda_{k}\left(  G\right)  \leq\lambda_{k}^{\ast}\left(  G\right)  \leq
\frac{n}{2\sqrt{k-1}},
\]
and
\[
\left\vert \lambda_{n-k+2}\left(  G\right)  \right\vert \leq\lambda_{k}^{\ast
}\left(  G\right)  \leq\frac{n}{2\sqrt{k-1}}.
\]

\end{theorem}

Further, letting $n\rightarrow\infty$, we get the bounds%
\begin{equation}
c_{k}\leq\frac{1}{2\sqrt{k-1}},\text{ \ \ }c_{-k+1}\leq\frac{1}{2\sqrt{k-1}%
},\text{ \ \ and \ \ \ }c_{k}^{\ast}\leq\frac{1}{2\sqrt{k-1}}.\label{ubc}%
\end{equation}
Simple as they are, bounds (\ref{ubc}) give the correct rate of growth of
$c_{k},$ $c_{-k}$ and $c_{k}^{\ast}$ in $k;$ in particular, the bound on
$c_{k}^{\ast}$ is quite tight. Note also that if $k=2,$ then equality holds in
each of the bounds (\ref{ubc}); on the other hand, if $k\geq3,$ the bound on
$c_{k}^{\ast}$ is attained for infinitely many $k,$ but the bounds on $c_{k}$
and $c_{-k+1}$ are never attained.

\begin{theorem}
\label{thub}If $k\geq3,$ then there is an $\varepsilon_{k}>0$ such that
\[
c_{k}<\frac{1}{2\sqrt{k-1}}-\varepsilon_{k}\text{ \ \ and\ \ \ }c_{-k+1}%
<\frac{1}{2\sqrt{k-1}}-\varepsilon_{k}.
\]

\end{theorem}

Our proof of Theorem \ref{thub} is quite complicated and uses the Removal
Lemma of Alon, Fischer, Krivelevich, and Szegedy \cite{AFKS00}, together with
other tools of analytic graph theory. Due to its length, it will not be given
in the present paper.

Although Theorem \ref{thub} may cast doubts as to the tightness of the bounds
(\ref{ubc}), we shall show that they can be matched by close lower bounds.

\subsection{Lower bounds on $c_{k}$ and $c_{-k}$}

In \cite{Tay71,Tay77}, Taylor came up with a remarkable class of strongly
regular graphs, which we shall use in several ways to give lower bounds on
$\lambda_{k}\left(  n\right)  $ and $\lambda_{-k}\left(  n\right)  .$ It seems
that Taylor's strongly regular graphs are a cornerstone in spectral graph
theory, and need to be known better. For reader's sake, in Section \ref{NBS}
gives a brief discussion on Taylor's graphs and their complements.

Note that Taylor's graphs contain roughly half of the total number of edges,
and the same holds for their complements. At the same time, almost all
eigenvalues of a Taylor graph are positive, and almost all eigenvalues of its
complement are negative. This combination of properties makes Taylor's graphs
and their complements very suitable for lower bounds on $c_{k}$ and $c_{-k}.$

To begin with, in the following theorem, we shall use Taylor's graphs to show
the tightness of the bounds (\ref{ubc}) for infinitely many, albeit handpicked
values of $k.$ The proof of Theorem \ref{th2} is in Section \ref{ps}.

\begin{theorem}
\label{th2}If $q$ is an odd prime power and $k$ $=q^{2}-q+1,$ then%
\begin{equation}
c_{k}>\frac{1}{2\sqrt{k-1}+1}\text{ \ and \ \ \ }c_{-k+1}>\frac{1}{2\sqrt
{k-1}+1}. \label{lbc}%
\end{equation}

\end{theorem}

Unfortunately, the bounds (\ref{lbc}) seem to be close to the best ones that
Taylor graphs can provide. Nevertheless, in the following theorem, we shall
use Taylor graphs to provide general asymptotics of $c_{k}$ and $c_{-k}$ for
any $k:$

\begin{theorem}
\label{thglb}There exists $k_{0}$ such that if $k>k_{0}$, then%
\begin{equation}
c_{k}>\frac{1}{2\sqrt{k-1}+\sqrt[3]{k}},\text{ \ \ and \ \ }c_{-k+1}>\frac
{1}{2\sqrt{k-1}+\sqrt[3]{k}}. \label{glb}%
\end{equation}
\textbf{ }
\end{theorem}

Theorem \ref{thglb}, whose proof is in Section \ref{ps}, shows that the upper
bounds (\ref{ubc}) on $c_{k}$ and $c_{-k}$ are asymptotically tight, although
there is a lot to improve. A particularly weak point of  bounds (\ref{glb}) is
the fact that $k_{0}$ is not known explicitly, due to a number-theoretic
result used in the proof. Below we provide a weaker theorem, with explicit
bounds. It also shows that the bound (\ref{Powbo}) fails for any $k\geq5$ and
$n$ sufficiently large.

\begin{theorem}
\label{thSB}If $5\leq k\leq15,$ then
\[
c_{k}\geq\frac{1}{k-1/2}\text{ \ \ and \ \ }c_{-k+1}\geq\frac{1}{k-1/2}.
\]

If $k\geq16,$ then
\[
c_{k}\geq\frac{1}{4\sqrt{k-1}}\text{ \ \ and \ \ }c_{-k+1}\geq\frac{1}%
{4\sqrt{k-1}}.
\]

\end{theorem}

Theorem \ref{thSB}, whose proof is also in Section \ref{ps}, leaves the
following two questions open:$\medskip$

\begin{question}
Is it true that $c_{3}=1/3?$ \medskip
\end{question}

\begin{question}
Is it true that $c_{4}=1/4?$
\end{question}

\subsection{Bounds on $c_{k}^{\ast}$}

It turns out that $\lambda_{k}^{\ast}\left(  n\right)  $ and $c_{k}^{\ast}$
can be estimated with greater precision than $\lambda_{k}\left(  n\right)  $
and $c_{k}$. We start by establishing a crucial connection between
$c_{k}^{\ast}$ and the existence of certain symmetric $\left(  -1,1\right)
$-matrices. Thus, write $\mathbb{U}_{n}$ for the set of symmetric $\left(
-1,1\right)  $-matrices or order $n.$\medskip

\begin{theorem}
\label{thMx}If $A$ is a symmetric $\left(  -1,1\right)  $-matrix of order $n$,
with $\lambda_{k}^{\ast}\left(  A\right)  =n/\sqrt{k},$ then
\[
c_{k+1}^{\ast}=\frac{1}{2\sqrt{k}}.
\]

\end{theorem}

\begin{proof}
Let $A=\left[  a_{i,j}\right]  \in\mathbb{U}_{n},$ with $\lambda_{k}^{\ast
}\left(  A\right)  =n/\sqrt{k}.$ Since%
\[
\sum_{i=1}^{k}\left(  \lambda_{i}^{\ast}\left(  A\right)  \right)  ^{2}\geq
k\left(  n/\sqrt{k}\right)  ^{2}=n^{2}=\sum_{i=1}^{n}\sum_{l=1}^{n}a_{i,j}%
^{2}=\sum_{i=1}^{n}\left(  \lambda_{i}^{\ast}\left(  A\right)  \right)  ^{2},
\]
we see that
\[
\lambda_{1}^{\ast}\left(  A\right)  =\cdots=\lambda_{k}^{\ast}\left(
A\right)  =n/\sqrt{k}\text{ \ \ and \ \ \ }\lambda_{i}^{\ast}\left(  A\right)
=0\text{ for }k<i\leq n\text{.}%
\]
Define a matrix $A^{\prime}\in\mathbb{U}_{2n}$ by%
\[
A^{\prime}=\left[
\begin{array}
[c]{cc}%
A & -A\\
-A & A
\end{array}
\right]  ,
\]
and note that all rowsums of $A^{\prime}$ are zero; therefore $0$ is an
eigenvalue of $A^{\prime}$\ with eigenvector\ $\mathbf{j}_{2n}.$ Now, define a
symmetric $\left(  0,1\right)  $-matrix $B$ by
\[
B=\frac{1}{2}\left(  A^{\prime}\otimes J_{t}+J_{2nt}\right)  .
\]
The order of $B$ is $2nt.$ Obviously $\lambda_{1}^{\ast}\left(  B\right)  =nt$
and
\begin{align*}
\lambda_{2}^{\ast}\left(  B\right)   &  =\cdots=\lambda_{k+1}^{\ast}\left(
B\right)  =nt/\sqrt{k},\\
\lambda_{i}^{\ast}\left(  B\right)   &  =0,\text{ \ \ }i=k+1,\ldots,nt.
\end{align*}
Now, zero the diagonal of $B$ and write $A$ for the resulting matrix. Clearly
$A$ is the adjacency matrix of some graph $G$ of order $2nt.$ Using Weyl's
inequalities (see Proposition \ref{proW0}), we find that
\begin{align*}
\lambda_{k+1}^{\ast}\left(  G\right)   &  =\lambda_{k+1}^{\ast}\left(
B-\left(  B-A\right)  \right)  \geq\lambda_{k+1}^{\ast}\left(  B\right)
-\lambda_{1}^{\ast}\left(  B-A\right)  \\
&  \geq\lambda_{k+1}^{\ast}\left(  B\right)  -\lambda_{1}^{\ast}\left(
I_{2nt}\right)  \\
&  =\frac{2nt}{2\sqrt{k}}-1.
\end{align*}
Letting $t\rightarrow\infty,$ we get%
\[
c_{k+1}^{\ast}=\frac{1}{2\sqrt{k}},
\]
completing the proof.
\end{proof}

Theorem \ref{thMx} motivates the introduction of a class of symmetric $\left(
-1,1\right)  $-matrices, which extend symmetric Hadamard matrices in a natural way.

\subsection{An extension of symmetric Hadamard matrices}

Write $n\left(  A\right)  $ for the order of a square matrix $A,$ and let
$\mathbb{S}_{k}$ be the set of symmetric $\left(  -1,1\right)  $-matrices with
$\lambda_{k}^{\ast}\left(  A\right)  =n\left(  A\right)  /\sqrt{k}.$ Note that
$\mathbb{S}_{k}$ can contain matrices of different order; for that matter, if
$\mathbb{S}_{k}$ is nonempty, then it is infinite. Note also that if $A\in$
$\mathbb{S}_{k},$ and a matrix $B$ can be obtained by permutations or
negations performed simultaneously on rows and columns of $A,$ then $B\in$
$\mathbb{S}_{k}$ as well; the reason is that singular values are not affected
by such operations.

As in the proof of Theorem \ref{thMx}, one can check the validity of the
following statement.

\begin{proposition}
\label{proma}A symmetric $\left(  -1,1\right)  $-matrix $A$ belongs to
$\mathbb{S}_{k}$ if and only if
\[
\lambda_{1}^{\ast}\left(  A\right)  =\cdots=\lambda_{k}^{\ast}\left(
A\right)  =n\left(  A\right)  /\sqrt{k},\text{ \ \ and \ \ }\lambda_{i}^{\ast
}\left(  A\right)  =0\text{\ for }k<i\leq n\left(  A\right)  \text{.}%
\]

\end{proposition}

Here is a summary of some properties of $\mathbb{S}_{k}$:\medskip

\qquad(1) If $A\in\mathbb{S}_{k}$ then $-A\in\mathbb{S}_{k};$

\qquad(2) If $A$ is a symmetric $\left(  -1,1\right)  $-matrix of rank 1, then
$A\in\mathbb{S}_{1};$ thus, $J_{n}\in\mathbb{S}_{1};$

\qquad(3) If $H$ is a symmetric Hadamard matrix of order $k,$ then
$H\in\mathbb{S}_{k};$

\qquad(4) If $A\in\mathbb{S}_{k}$ and $B\in\mathbb{S}_{l},$ then $A\otimes
B\in\mathbb{S}_{kl};$ hence, if $\mathbb{S}_{k}\neq\varnothing,$ then
$\mathbb{S}_{2k}\neq\varnothing;$

\qquad(5) If $A\in\mathbb{S}_{k}$, then $A\otimes J_{n}$ $\in\mathbb{S}_{k}$
for any $n\geq1;$ hence $\mathbb{S}_{k}$ is infinite;

\qquad(6) If $\mathbb{S}_{k}\neq\varnothing$, then $\mathbb{S}_{k}$ contains
matrices with all their rowsums equal to zero.\medskip

We omit the proofs of (1)-(5), but here is sketch of a proof of (6):\emph{
}the rank of the matrix%
\[
K=\left[
\begin{array}
[c]{cc}%
1 & -1\\
-1 & 1
\end{array}
\right]
\]
is $1$, so $K\in S_{1}.$ Now, if $A\in S_{k},$ then $K\otimes A\in S_{k},$ and
the rowsums of $K\otimes A$ are zero.\emph{\medskip}

Properties (1)-(6) allow to show that $\mathbb{S}_{k}$ contains matrices with
some special properties, as in the following proposition:

\begin{proposition}
\label{props}If $A\in\mathbb{S}_{k},$ then there is a $B\in\mathbb{S}_{2k}$
such that:

(i) $B$ has exactly $k$ positive and exactly $k$ negative eigenvalues;

(ii) the rowsums of $B$ are equal to $0.$
\end{proposition}

To check this proposition, set
\[
H_{2}=\left[
\begin{array}
[c]{cc}%
1 & 1\\
1 & -1
\end{array}
\right]  ,
\]
and let $B=K\otimes\left(  H_{2}\otimes A\right)  ;$ obviously $B\in$
$\mathbb{S}_{2k},$ and $B$ satisfies \emph{(i) }and\emph{ (ii)}.\bigskip

The principal question about $\mathbb{S}_{k}$ is the following one:

\begin{problem}
\label{probS}For which $k$ is $\mathbb{S}_{k}$ nonempty?
\end{problem}

Below we shall show that $\mathbb{S}_{k}$ is empty if $k$ is odd and is not a
square. On the positive side, Property (3) implies that there are infinitely
many $k$ for which $\mathbb{S}_{k}$ is nonempty. In particular, Paley's
construction of symmetric Hadamard matrices (see, e.g., \cite{CrKh07}) implies
the following fact:

\begin{proposition}
If $p$ is a prime power and $p=1$ $\operatorname{mod}$ $4,$ then
$\mathbb{S}_{2\left(  p+1\right)  }$ is not empty.
\end{proposition}

We shall prove more definite assertions about $\mathbb{S}_{k}$ by using the
fact that the singular values of a real symmetric matrix are the absolute
values of its eigenvalues.

\begin{proposition}
\label{proq}If $A\in\mathbb{S}_{k},$ then either $k$ is an exact square, or
$A$ has the same number of positive and negative eigenvalues.
\end{proposition}

\begin{proof}
Let $A=\left[  a_{i,j}\right]  \in\mathbb{S}_{k}.$ Setting $l=\lambda
_{1}^{\ast}\left(  A\right)  $ and $n=n\left(  A\right)  ,$ we have
\[
kl^{2}=\sum_{\lambda_{i}\left(  A\right)  >0}\lambda_{i}^{2}\left(  A\right)
+\sum_{\lambda_{i}\left(  A\right)  <0}\lambda_{i}^{2}\left(  A\right)
=\mathrm{tr}\text{ }A^{2}=\sum_{i=1}^{n}\sum_{l=1}^{n}a_{i,j}^{2}=n^{2}.
\]
On the other hand, writing $n_{+}$ and $n_{-}$ for the number of positive and
negative eigenvalues of $A,$ we see that%
\[
\left(  n_{+}-n_{-}\right)  l=\sum_{\lambda_{i}\left(  A\right)  >0}%
\lambda_{i}\left(  A\right)  +\sum_{\lambda_{i}\left(  A\right)  <0}%
\lambda_{i}\left(  A\right)  =\mathrm{tr}\text{ }A.
\]
Since $\mathrm{tr}$ $A$ is an integer, either $l$ is rational or $n_{+}=n_{-}%
$. Since $l=n/\sqrt{s}$, it may be rational only if $\sqrt{s}$ is an integer,
completing the proof.
\end{proof}

\begin{corollary}
If $k$ is odd and $\mathbb{S}_{k}$ is nonempty, then $k$ is an exact square.
\end{corollary}

It is interesting to see what we can say about $\mathbb{S}_{k}$ for small $k,$
say for $k\leq10.$ First, the above corollary implies that $\mathbb{S}_{k}$ is
empty for $k=3,5,$ and $7.$ Below we shall show that the converse of
Proposition \ref{proq} is partially true as well, that is to say: \emph{if
}$k$\emph{ is an exact square, then }$S_{k}$\emph{ is nonempty.} Thus, in view
of properties (1)-(5), we see that $\mathbb{S}_{k}$ is nonempty for
$k=1,2,4,8,$ and $9.$ The first unknown cases are $k=6$ and $k=10.$

\begin{question}
Is $\mathbb{S}_{6}$ empty?
\end{question}

\medskip

\subsection{\label{cons}Constructions of matrices in \ $\mathbb{S}_{s^{2}}$}

In this subsection we shall construct classes of symmetric $\left(
-1,1\right)  $-matrices, which show that $\mathbb{S}_{k}$ is nonempty if $k$
is an exact square or is the double of an exact square. Except for this
primary goal, our matrices must have other specific properties, which are
necessary for subsequent applications; thus the statements are somewhat involved.

The proofs of Theorems \ref{thKHN}, \ref{thj}, and \ref{thj1} are variations
of Kharaghani's method \cite{Kha85} for constructing Hadamard matrices. Our
approach suggests that this method is generic and can be used to construct
$\left(  -1,1\right)  $-matrices that are more general than Hadamard matrices.

The proofs of Theorems \ref{thKHN}, \ref{thj}, and \ref{thj1} are in Section
\ref{ps}.

\begin{theorem}
\label{thKHN}For any integer $s\geq2$, there are an integer $n\geq s$ and a
symmetric $\left(  -1,1\right)  $-matrix $B$ of order $ns$ such that:

(i) $B$ has exactly $s^{2}$ nonzero eigenvalues, of which $\binom{s-1}{2}$ are
equal to $n,$ and $\binom{s+1}{2}$ are equal to $-n$;

(ii) the rowsums of $B$ are equal to $0;$

(iii) the diagonal entries of $B$ are equal to $-1.$
\end{theorem}

Note that clause \emph{(i)} of Theorem \ref{thKHN} implies that $B\in
\mathbb{S}_{s^{2}}.$ Further, negating the matrix $B,$ we get the following variation.

\begin{corollary}
For any integer $s\geq2$, there are an integer $n\geq s$ and a symmetric
$\left(  -1,1\right)  $-matrix $B$ of order $ns$ such that:

(i) $B$ has exactly $s^{2}$ nonzero eigenvalues, of which $\binom{s+1}{2}$ are
equal to $n,$ and $\binom{s-1}{2}$ are equal to $-n$;

(ii) the rowsums of $B$ are equal to $0;$

(iii) the diagonal entries of $B$ are equal to $1.$
\end{corollary}

Theorem \ref{thKHN} will be used to give some answers to Problem \ref{mpro}.
For other purposes we shall need two other theorems.

\begin{theorem}
\label{thj}For any integer $s\geq2$, there are an integer $n\geq s$ and a
symmetric $\left(  -1,1\right)  $-matrix $B$ of order $ns$ such that:

(i) $B$ has exactly $s^{2}$ nonzero eigenvalues, of which $\binom{s+1}{2}$ are
equal to $n,$ and $\binom{s-1}{2}$ are equal to $-n$;

(ii) the vector $\mathbf{j}_{ns}$ is an eigenvector of $B$ to the eigenvalue
$-n$;

(iii) the diagonal entries of $B$ are equal to $1.$
\end{theorem}

\begin{theorem}
\label{thj1}For any integer $s\geq2$, there are an integer $n\geq s$ and a
symmetric $\left(  -1,1\right)  $-matrix $B$ of order $ns$ such that:

(i) $B$ has exactly $s^{2}$ nonzero eigenvalues, of which $\binom{s+1}{2}-1$
are equal to $n,$ and $\binom{s-1}{2}+1$ are equal to $-n$;

(ii) all rowsums of $B$ are equal to $-n$.
\end{theorem}

Note that Theorem \ref{thKHN} and Property (2) imply that for any natural
number $s$, the classes $\mathbb{S}_{s^{2}}$ and $\mathbb{S}_{2s^{2}}$ are
nonempty. Before continuing, let us state some explicit solutions to Problem
\ref{probS}:

\begin{proposition}
\label{propE}Let $s$ be a natural number and $p$ be a prime power, with $p=1$
$\operatorname{mod}$ $4.$ Then the classes $\mathbb{S}_{s^{2}},$
$\mathbb{S}_{2s^{2}},$ $\mathbb{S}_{2s^{2}\left(  p+1\right)  }$ and
$\mathbb{S}_{4s^{2}\left(  p+1\right)  }$ are nonempty.
\end{proposition}

\subsection{Finding $\lambda_{k}^{\ast}\left(  n\right)  $ and $c_{k}%
^{\ast\text{ }}$for infinitely many $k$ and $n$}

Observe that using the classes $\mathbb{S}_{k},$ Theorem \ref{thMx} can be
restated as:\emph{ if }$\mathbb{S}_{k}$\emph{ is nonempty, then }%
$c_{k+1}^{\ast}=1/(2\sqrt{k}).$ Proposition \ref{propE} and properties (1)-(6)
give many values of $k$ for which this equality holds. Moreover, using the
finer properties outlined in Theorem \ref{thKHN}, we can determine the exact
value of $\lambda_{s^{2}+1}^{\ast}\left(  n\right)  $ for infinitely many $n.$

\begin{theorem}
\label{thp}If $s\geq1,$ then there is an integer $n>s,$ such that for every
integer $t\geq1,$ there is a graph $G$ of order $snt$ with\emph{ }%
\[
\lambda_{s^{2}+1}^{\ast}\left(  G\right)  =\frac{nt}{2}.
\]

\end{theorem}

\begin{proof}
Let $n\geq s,$ and let $B$ be the matrix of order $sn$ constructed in Theorem
\ref{thKHN}. Set
\[
C=\frac{1}{2}\left(  B+J_{sn}\right)  ,
\]
and note that $C$ is a symmetric $\left(  0,1\right)  $-matrix with zero
diagonal. Since all rowsums of $B$ are zero, the vector $\mathbf{j}_{sn}$ is
an eigenvector of $B$ to the eigenvalue $0;$ thus all eigenvectors to nonzero
eigenvalues of $B$ are orthogonal to $\mathbf{j}_{sn}.$ Obviously every
eigenvector of $B$ is an eigenvector to $C.$ Thus, $sn/2$ is an eigenvalue of
$C;$ also, $C$ has $\binom{s-1}{2}$ eigenvalues equal to $n/2$ and
$\binom{s+1}{2}$ eigenvalues equal to $-n/2;$ the remaining eigenvalues of $C$
are zero. Now, letting $A=C\otimes J_{t},$ one sees that $A$ is the adjacency
matrix of a graph $G$ of order $snt$ with
\[
\lambda_{s^{2}+1}^{\ast}\left(  G\right)  =\frac{1}{2}\lambda_{s^{2}}^{\ast
}\left(  B\right)  t=\frac{snt}{2s},
\]
completing the proof of Theorem \ref{thp}.
\end{proof}

Theorem \ref{thKHN} helps also to find concise asymptotics of $c_{k}^{\ast}.$
Indeed, since $c_{k}^{\ast}$ is nonincreasing in $k,$ letting $s$ to be the
smallest integer such that $s^{2}+1\geq k,$ we see that $c_{k}^{\ast}%
\geq1/\left(  2s\right)  ;$ since $\left(  s-1\right)  ^{2}+1<k,$ we get the
following theorem:

\begin{theorem}
\label{thga}For any $k\geq3,$%
\[
\frac{1}{2\sqrt{k-1}}\geq c_{k}^{\ast}>\frac{1}{2\sqrt{k-1}+2}=\frac{1}%
{2\sqrt{k-1}}+O\left(  k^{-1}\right)  .
\]

\end{theorem}

It seems quite clear that the lower bound on $c_{k}^{\ast}$ can be improved,
so we raise the following problem.

\begin{problem}
Is there a positive constant $C$ such that for any $k\geq3,$
\[
c_{k}^{\ast}>\frac{1}{2\sqrt{k+C}}?
\]

\end{problem}

\section{\label{NGs}Spectral Nordhaus-Gaddum problems}

It turns out that the classes $\mathbb{S}_{k}$ help to find infinitely many
solutions to a general spectral Nordhaus-Gaddum problem. Nordhaus-Gaddum
problems in general, form a notable part of extremal graph theory, see, e.g.,
the recent survey \cite{AoHa13} for their numerous variations. In particular,
spectral Nordhaus-Gaddum problems have been studied first in 1970, by Nosal
\cite{Nos70}, and have attracted a lot of attention since then.

Thus, let $\overline{G}$ denote the complement of a graph $G.$ Given $n\geq
k\geq1,$ define the functions $f_{k}\left(  n\right)  ,$ $f_{-k}\left(
n\right)  ,$ and $f_{k}^{\ast}\left(  n\right)  $ as
\begin{align*}
f_{k}\left(  n\right)   &  =\max_{v\left(  G\right)  =n}\lambda_{k}\left(
G\right)  +\lambda_{k}(\overline{G}),\\
f_{-k}\left(  n\right)   &  =\max_{v\left(  G\right)  =n}\left\vert
\lambda_{n-k+1}\left(  G\right)  \right\vert +|\lambda_{n-k+1}(\overline
{G})|,\\
f_{k}^{\ast}\left(  n\right)   &  =\max_{v\left(  G\right)  =n}\lambda
_{k}^{\ast}\left(  G\right)  +\lambda_{k}^{\ast}(\overline{G}).
\end{align*}
Clearly, we define $f_{k}\left(  n\right)  ,$ $f_{-k}\left(  n\right)  ,$ and
$f_{k}^{\ast}\left(  n\right)  $ similarly to $\lambda_{k}\left(  n\right)  ,$
$\lambda_{-k}\left(  n\right)  ,$ and $\lambda_{k}^{\ast}\left(  n\right)  .$
Note that $f_{k}^{\ast}\left(  n\right)  $ is a new function, but
$f_{k}\left(  n\right)  $ and $f_{-k}\left(  n\right)  $ have been introduced
in \cite{Nik07} with different, albeit essentially equivalent definitions.

Now, let us reiterate and extend a problem raised in \cite{Nik07}:

\begin{problem}
\label{NGpro}For any $k\geq1,$ find $f_{k}\left(  n\right)  ,$\ $f_{-k}\left(
n\right)  ,$ and $f_{k}^{\ast}\left(  n\right)  .$
\end{problem}

The function $f_{1}\left(  n\right)  $ has been studied by Nosal \cite{Nos70}:
finding it have turned out to be a hard problem, which has been resolved only
recently, in \cite{Csi09} and \cite{Ter11}. This case sticks out from the
rest, both with its particular extremal graphs, as with its particular
methods. The first "mainstream" case is $f_{2}\left(  n\right)  ,$ which has
been determined in \cite{Nik07}.

Recently, in \cite{NiYu14}, two tight upper bounds have been given:

\emph{If }$k\geq2$\emph{ and }$n\geq15\left(  k-1\right)  ,$\emph{ then }%
\begin{equation}
f_{k}\left(  n\right)  \,\leq\frac{n}{\sqrt{2\left(  k-1\right)  }%
}-1.\label{bo2}%
\end{equation}
\emph{Likewise, if }$k\geq1$ \emph{and} $n\geq4^{k},$\emph{ then }%
\begin{equation}
f_{-k}\left(  n\right)  \,\leq\frac{n}{\sqrt{2k}}+1.\label{bo3}%
\end{equation}
In turns out that the bounds (\ref{bo2}) and (\ref{bo3}) capture the rate of
growth of $f_{k}\left(  n\right)  $\ and $f_{-k}\left(  n\right)  $ pretty
tightly. In \cite{NiYu14}, it has been shown that (\ref{bo2}) and (\ref{bo3})
are essentially best possible if $k=2^{s-1}+1$ and $s=2,3,\ldots.$ What is
more, for such $k$ it has proved that
\[
f_{k}\left(  n\right)  \geq\frac{n}{\sqrt{2\left(  k-1\right)  }}-2\text{
\ \ and \ \ \ }f_{-k}\left(  n\right)  \geq\frac{n}{\sqrt{2k}}.
\]
for infinitely many $n.\medskip$

The study of $f_{k}\left(  n\right)  ,$ $f_{-k}\left(  n\right)  ,$ and
$f_{k}^{\ast}\left(  n\right)  $ can be put on the same ground as $\lambda
_{k}\left(  n\right)  ,$ $\lambda_{-k}\left(  n\right)  ,$ and $\lambda
_{k}^{\ast}\left(  n\right)  .$ First, using the methods of \cite{Nik06}, one
can show that the limits
\[
f_{k}=\lim_{n\rightarrow\infty}\frac{f_{k}\left(  n\right)  }{n}%
,\text{\ \ \ \ }f_{-k}=\lim_{n\rightarrow\infty}\frac{f_{-k}\left(  n\right)
}{n},\text{ \ \ and \ \ \ }f_{k}^{\ast}=\lim_{n\rightarrow\infty}\frac
{f_{k}^{\ast}\left(  n\right)  }{n}%
\]
exist, and the following inequalities hold for every $n:$%
\[
f_{k}\left(  n\right)  \leq f_{k}n-1,\text{ \ \ }f_{-k}\left(  n\right)  \leq
f_{-k}n+1,\text{\ \ \ \ and \ \ \ \ }f_{k}^{\ast}\left(  n\right)  \leq
f_{k}^{\ast}n+1.
\]
Now, the constructions developed in Section \ref{cons} allow to resolve
Problem \ref{NGpro} for infinitely many values of $k.$ Indeed, it turns out
that if \ $\mathbb{S}_{k}$ is nonempty, then
\[
f_{k+1}=\frac{1}{\sqrt{2k}}\text{ \ \ and \ \ \ }f_{-k}=\frac{1}{\sqrt{2k}}.
\]
These equalities follow from the theorem below, where we prove a more precise
result, involving $n$ as well.

\begin{theorem}
\label{thNG}If $\mathbb{S}_{k}$ is nonempty, then there is an integer $n>k,$
such that for every integer $t\geq1,$ there is a graph $G$ of order $nt$ with
\begin{equation}
f_{k+1}\left(  n\right)  \geq\frac{nt}{\sqrt{2k}}-2 \label{in1}%
\end{equation}
and%
\begin{equation}
f_{-k}\left(  n\right)  \geq\frac{nt}{\sqrt{2k}}. \label{in2}%
\end{equation}

\end{theorem}

\begin{proof}
If $\mathbb{S}_{k}$ is nonempty, Proposition \ref{props} implies that there
exists $B\in\mathbb{S}_{2k},$ say of order $n,$ such that,
\begin{align*}
\lambda_{1}\left(  B\right)   &  =\cdots=\lambda_{k}\left(  B\right)
=\frac{n}{\sqrt{2k}},\\
\lambda_{n-k+1}\left(  B\right)   &  =\cdots=\lambda_{n}\left(  B\right)
=-\frac{n}{\sqrt{2k}},\\
\lambda_{i}^{\ast}\left(  B\right)   &  =0\text{ for }k<i\leq n-k\text{.}%
\end{align*}
Define a matrix $A^{\prime}$ by%
\[
A^{\prime}=\frac{1}{2}\left(  B\otimes J_{t}\right)  +J_{nt}.
\]
Now, zero the diagonal of $A^{\prime}$ and write $A$ for the resulting matrix.
Note that $A$ is a symmetric $\left(  0,1\right)  $ matrix with zero diagonal,
so it is the adjacency matrix of graph $G$ of order $nt.$ As in the proof of
Theorem \ref{thMx}, we see that
\[
\lambda_{k+1}\left(  G\right)  \geq\frac{nt}{2\sqrt{2k}}-1.
\]
On the other hand, taking the matrix
\[
\overline{A}^{\prime}=\frac{1}{2}\left(  -B\otimes J_{t}\right)  +J_{nt},
\]
and zeroing its main diagonal, we obtain the matrix $\overline{A},$ which is
obviously the adjacency matrix of the complement of $G.$ Like above we have,
\[
\lambda_{k+1}(\overline{G})\geq\frac{nt}{2\sqrt{2k}}-1,
\]
and inequality (\ref{in1}) follows. The proof of (\ref{in2}) is similar and is omitted.
\end{proof}

In view of Theorem \ref{thNG} and properties (1)-(6), we get numerous examples
for which bounds (\ref{bo2}) and (\ref{bo3}) are essentially best; we refer to
Proposition \ref{propE} for some explicit values.

We finish the discussion of $f_{k}\left(  n\right)  $ and $f_{-k}\left(
n\right)  $ with general asymptotics of $f_{k}$ and $f_{-k}.$ Since $f_{k}$
and $f_{-k}\left(  n\right)  $ are nonincreasing in $k$, letting $s$ to be the
smallest integer such that $s^{2}\geq k,$ we see that
\[
f_{k+1}\geq\frac{1}{\sqrt{2}s}\text{ \ \ and \ \ }f_{-k}>\frac{1}{\sqrt{2}s};
\]
since $\left(  s-1\right)  ^{2}<k,$ we get the following theorem:

\begin{theorem}
For any $k\geq2,$%
\[
\frac{1}{\sqrt{2\left(  k-1\right)  }}\geq f_{k}>\frac{1}{\sqrt{2\left(
k-1\right)  }+\sqrt{2}}=\frac{1}{\sqrt{2\left(  k-1\right)  }}+O\left(
k^{-1}\right)  ,
\]
and%
\[
\frac{1}{\sqrt{2\left(  k-1\right)  }}\geq f_{-k+1}>\frac{1}{\sqrt{2\left(
k-1\right)  }+\sqrt{2}}=\frac{1}{\sqrt{2\left(  k-1\right)  }}+O\left(
k^{-1}\right)  .
\]

\end{theorem}

Finally, let us briefly discuss the function $f_{k}^{\ast}\left(  n\right)  ,$
which is somewhat easier to deal with, and can be derived mainly from
$\lambda_{k}^{\ast}\left(  n\right)  $. First, Theorem \ref{th1} implies
immediately that
\begin{equation}
f_{k}^{\ast}\left(  n\right)  \leq2\lambda_{k}^{\ast}\left(  n\right)
\leq\frac{n}{\sqrt{k-1}}.\label{bo4}%
\end{equation}
This easy bound is rather different from (\ref{bo2}) and (\ref{bo3}), whose
proofs are much subtler anyway. Nonetheless, bound (\ref{bo4}) gives the
correct rate of growth of $f_{k}^{\ast}\left(  n\right)  .$ Using Theorems
\ref{thMx} and \ref{thga}, we immediately come up with the following statements:

\begin{theorem}
If $\mathbb{S}_{k}$ is nonempty, then there is an integer $n>k,$ such that for
every integer $t\geq1,$ there is a graph $G$ of order $nt$ with
\[
f_{k+1}^{\ast}\left(  n\right)  \geq\frac{nt}{\sqrt{k}}-2
\]

\end{theorem}

\begin{theorem}
For any $k\geq2,$%
\[
\frac{1}{\sqrt{k-1}}\geq f_{k}^{\ast}>\frac{1}{\sqrt{k-1}+1}=\frac{1}%
{\sqrt{k-1}}+O\left(  k^{-1}\right)  .
\]

\end{theorem}

\section{\label{KFs}Sums of eigenvalues and sums of singular values}

In addition to individual eigenvalues and singular values of graphs, it is of
interest to consider certain sums thereof. In particular, let
\begin{align*}
\tau_{k}\left(  n\right)   &  =\max_{v\left(  G\right)  =n}\lambda_{1}\left(
G\right)  +\cdots+\lambda_{k}\left(  G\right)  ,\\
\xi_{k}\left(  n\right)   &  =\max_{v\left(  G\right)  =n}\lambda_{1}^{\ast
}\left(  G\right)  +\cdots+\lambda_{k}^{\ast}\left(  G\right)  .
\end{align*}
Note that $\xi_{k}\left(  n\right)  $ is the maximal \emph{Ky Fan }%
$k$\emph{-norm} of a graph of order $n$. In particular, the Ky Fan $n$-norm is
known as the \emph{trace norm} of $G,$ and has been extensively studied under
the name \emph{graph energy,} a concept introduced by Gutman in \cite{Gut78};
see also \cite{GLS12} for the current state of this research. Note that,
$\xi_{n}\left(  n\right)  $ is just the maximum energy of a graph of order
$n,$ which also has been studied, see, e.g., \cite{Hae08}, \cite{KoMo01}, and
\cite{Nik07j}.

The research on graph energy is truly monumental, but with the flexibility of
the parameter $k,$ the Ky Fan $k$-norms offer a considerably vaster
playground. Here we shall focus only on the following principal question,
raised in \cite{Ni11}:

\begin{problem}
\label{MNpro}For any $k\geq1,$ find $\tau_{k}\left(  n\right)  $\ and $\xi
_{k}\left(  n\right)  .$
\end{problem}

For a start, note that the inequality $\lambda_{i}\left(  G\right)
\leq\lambda_{i}^{\ast}\left(  G\right)  $ implies that $\tau_{k}\left(
n\right)  \leq\xi_{k}\left(  n\right)  $ for any $k$ and $n.$ For general
$k>2,$ the function $\tau_{k}\left(  n\right)  $ has been studied by Mohar in
\cite{Moh09}; and in turn, $\xi_{k}\left(  n\right)  $ has been studied by the
author in \cite{Ni11}. The exact values of $\tau_{k}\left(  n\right)  $ and
$\xi_{k}\left(  n\right)  $ are unknown for most values of $k;$ in particular,
neither $\tau_{2}\left(  n\right)  $ nor $\xi_{2}\left(  n\right)  $ are known
yet: see \cite{EMNA08} for $\tau_{2}\left(  n\right)  ,$ and \cite{GHK01} for
$\xi_{2}\left(  n\right)  .$ However, estimating $\tau_{k}\left(  n\right)  $
and $\xi_{k}\left(  n\right)  $ is possible for large $k$. Indeed, Mohar
\cite{Moh09} proved the asymptotics%
\begin{equation}
\frac{1}{2}\left(  \frac{1}{2}+\sqrt{k}-o\left(  k^{-2/5}\right)  \right)
<\frac{\tau_{k}\left(  n\right)  }{n}\leq\frac{1}{2}\left(  1+\sqrt{k}\right)
.\label{Mohbo}%
\end{equation}
Note the gap $1/2+o\left(  1\right)  $ between the upper and lower bounds in
(\ref{Mohbo}), which is very challenging to close. In general, finding
$\tau_{k}\left(  n\right)  $ seems a hard problem, a lot harder than finding
$\xi_{k}\left(  n\right)  $. In particular, it is easy to show that the limit
$\tau_{k}=\lim\limits_{n\rightarrow\infty}\tau_{k}\left(  n\right)  /n$ exists
for any fixed $k\geq1,$ but this limit is not known for any $k\geq2.$ Thus, we
suggest the following concrete conjecture:

\begin{conjecture}
For any $k\geq2,$ there is an $\varepsilon_{k}>0$ such that
\[
\tau_{k}<\frac{1}{2}\left(  1+\sqrt{k}-\varepsilon_{k}\right)  .
\]

\end{conjecture}

In contrast to $\tau_{k}\left(  n\right)  $, we shall find $\xi_{k}\left(
n\right)  $ for infinitely many values of $k$ and $n.$ To begin with, in
\cite{Ni11} it was shown that if $n\geq k\geq1,$ then
\begin{equation}
\xi_{k}\left(  n\right)  \leq\frac{1}{2}\left(  1+\sqrt{k}\right)
n,\label{Nikbo}%
\end{equation}
which strengthens the upper bound (\ref{Mohbo}). In fact, unlike $\tau
_{k}\left(  n\right)  ,$ the function $\xi_{k}\left(  n\right)  $ attains the
upper bound (\ref{Nikbo}) for infinitely many $k$ and $n.$ Indeed, let $H$ be
a symmetric regular Hadamard matrix of order $k,$ with positive rowsums, and
with $-1$ along the main diagonal. Then the matrix
\[
A=\frac{1}{2}\left(  H\otimes J_{n}\right)  +J_{kn}%
\]
is the adjacency matrix of a graph $G$ of order $kn,$ with
\[
\lambda_{1}^{\ast}\left(  G\right)  +\cdots+\lambda_{k}^{\ast}\left(
G\right)  =\frac{1}{2}\left(  1+\sqrt{k}\right)  n,
\]
and so $\xi_{k}\left(  n\right)  $ attains the upper bound (\ref{Nikbo}). It
is known that symmetric regular Hadamard matrix with equal rowsums and with
$-1$ along the main diagonal exist for $k=4m^{4}$ and any $m=1,2,\ldots$, see
\cite{HaXi09} for details. In fact, there are many more cases of $k$ for which
the upper bound (\ref{Nikbo}) is attained.\medskip

First, we shall show that if the bound (\ref{Nikbo}) is attained, then $k$ is
an exact square:

\begin{theorem}
\label{thk}If $G$ is a graph of order $n$ such that
\[
\lambda_{1}^{\ast}\left(  G\right)  +\cdots+\lambda_{k}^{\ast}\left(
G\right)  =\frac{1}{2}\left(  1+\sqrt{k}\right)  n,
\]
then $k$ is an exact square.
\end{theorem}

\begin{proof}
Write $\left\Vert A\right\Vert _{\ast k}$ the sum of the $k$ largest singular
values of $A$. Suppose that $G$ is a graph that satisfies the hypothesis and
write $A$ for its adjacency matrix.

Note that $J_{n}-2A$ is a symmetric $\left(  -1,1\right)  $-matrix and so, in
view of the AM-QM inequality, we find that%
\[
\sum_{i=1}^{k}\lambda_{i}^{\ast}\left(  J_{n}-2A\right)  \leq\sqrt{k\sum
_{i=1}^{k}\lambda_{i}^{\ast}{}^{2}\left(  J_{n}-2A\right)  }\leq\sqrt
{k\sum_{i=1}^{n}\lambda_{i}^{\ast}{}^{2}\left(  J_{n}-2A\right)  }=\sqrt{k}n.
\]
Therefore, using the the triangle inequality for the Ky Fan $k$-norm
$\left\Vert X+Y\right\Vert _{\ast k}\leq$ $\left\Vert X\right\Vert _{\ast
k}+\left\Vert Y\right\Vert _{\ast k}$ (see \cite{HoJo94}, p.196), we find
that
\[
\left(  1+\sqrt{k}\right)  n=2\left\Vert A\right\Vert _{\ast k}=\left\Vert
2A\right\Vert _{\ast k}\leq\left\Vert 2A-J_{n}\right\Vert _{\ast k}+\left\Vert
J_{n}\right\Vert _{\ast k}\leq\sqrt{k}n+n.
\]
Thus, equalities hold throughout the above line, and so, $2A-J_{n}$ has $k$
nonzero singular values, which are equal. We get
\[
\lambda_{k}^{\ast}\left(  J_{n}-2A\right)  =n/\sqrt{k},
\]
implying that $J_{n}-2A\in\mathbb{S}_{k}.$ On the other hand, $\mathrm{tr}%
\left(  J_{n}-2A\right)  =n\neq0,$ so $J_{n}-2A$ cannot have the same number
of positive and negative eigenvalues, and Proposition \ref{proq} implies that
$k$ is an exact square.
\end{proof}

The matrix built in Theorem \ref{thj} helps to prove that the converse of the
above theorem is partially true as well.

\begin{theorem}
\label{thck}Let $s$ be an even positive integer. There exists a positive
integer $n,$ such that for every positive integer $t,$ there is a graph $G$ of
order $snt,$ with
\[
\lambda_{1}^{\ast}\left(  G\right)  +\cdots+\lambda_{s^{2}}^{\ast}\left(
G\right)  =\frac{1}{2}\left(  1+s\right)  snt.
\]

\end{theorem}

Theorem \ref{thck} is proved in Section \ref{ps}. It is as good as one can
get, but we can prove it only if $s$ is even. If $s$ is odd, we can do just
slightly worse, showing that $\xi_{k}\left(  n\right)  $ is just below the
upper bound. To this effect, we shall prove a more general theorem, and deduce
this fact as a corollary. The proof of the theorem is in Section \ref{ps}.

\begin{theorem}
\label{thck1}Suppose that $\mathbb{S}_{k}$ contains a regular matrix $B$ with
nonzero rowsums, say of order $n$. Then for any positive integer $t,$ there is
a graph $G$ of order $nt$ with
\begin{equation}
\lambda_{1}^{\ast}\left(  G\right)  +\cdots+\lambda_{k}^{\ast}\left(
G\right)  \geq\frac{1}{2}\left(  1+\sqrt{k}\right)  nt-k. \label{in3}%
\end{equation}

\end{theorem}

Dividing both sides of (\ref{in3}) by $nt$ and letting $t\rightarrow\infty,$
we obtain the following corollary.

\begin{corollary}
If \ $\mathbb{S}_{k}$ contains a regular matrix $B$ with nonzero rowsums,
then
\[
\lim\limits_{n\rightarrow\infty}\frac{\xi_{k}\left(  n\right)  }{n}%
=\frac{1+\sqrt{k}}{2}.
\]

\end{corollary}

Let us note that the premise that $\mathbb{S}_{k}$ contains a regular matrix,
with nonzero rowsums is not difficult to satisfy. Indeed, Theorem \ref{thj1}
implies that for any integer $s\geq2,$ the set $\mathbb{S}_{s^{2}}$ contains a
regular matrix with nonzero rowsums.\medskip

Theorem \ref{thk} does not shed any light on the case when $k$ is not an exact
square, so we suggest the following concrete conjecture.

\begin{conjecture}
There exist infinitely many integers $k$ such that
\[
\lim\limits_{n\rightarrow\infty}\frac{\xi_{k}\left(  n\right)  }{n}%
<\frac{1+\sqrt{k}}{2}.
\]

\end{conjecture}

We end up this section with the easy asymptotics%
\[
\frac{\sqrt{k}}{2}\leq\lim\limits_{n\rightarrow\infty}\frac{\xi_{k}\left(
n\right)  }{n}\leq\frac{1+\sqrt{k}}{2},
\]
whose proof is omitted.

\section{\label{NBS}Notation, background, and support}

For graph notation and concepts undefined here, the reader is referred to
\cite{Bol98}. For general reference on graph spectra, see \cite{CRS10}; for
reference on Hadamard matrices and symmetric Latin squares, see \cite{CrKh07}
and \cite{IoSh06}; for reference on strongly regular graphs and their
eigenvalues, see \cite{GoRo01}.\medskip

We write $I_{n}$ and $J_{n}$ for the identity and the all ones matrix of order
$n.$ The $n$-dimensional vector of all ones is denoted by $\mathbf{j}_{n}.$ As
usual, the Kronecker product of two matrices $A$ and $B$ is denoted by
$A\otimes B.$ We recall that if $A$ and $B$ are square, then the spectrum of
$A\otimes B$ consists are all products of eigenvalues of $A$ and eigenvalues
of $B,$ with multiplicities counted. Also, the Kronecker product of symmetric
matrices is symmetric.\medskip

In this paper \emph{regular matrix} means a matrix whose rowsums are
equal.\medskip

Next, we shall give necessary details on Weyl's inequalities, graphs blowups,
Taylor strongly regular graphs, and symmetric Latin squares.\medskip

\subsection{\label{WS}Weyl's inequalities}

If $A$ is a Hermitian matrix of order $n,$ write $\lambda_{1}\left(  A\right)
,\ldots,\lambda_{n}\left(  A\right)  $ for its eigenvalues ordered as
$\lambda_{1}\left(  A\right)  \geq\cdots\geq\lambda_{n}\left(  A\right)  .$
Weyl proved the following useful inequalities for the eigenvalues of sums of
Hermitian matrices, (see, e.g. \cite{HoJo88}, p. 181):

Let $A$ and $B$ be Hermitian matrices of order $n,$ and let $1\leq i\leq n$
and $1\leq j\leq n.$ Then
\[
\lambda_{i}(A)+\lambda_{j}(B)\leq\lambda_{i+j-n}(A+B),\text{ if }i+j\geq n+1.
\]
\medskip

The following two immediate corollaries are used throughout the paper.

\begin{proposition}
\label{proW0}Suppose that $A^{\prime}$ is a symmetric $\left(  0,1\right)
$-matrix of order $n.$ If $A$ is the matrix obtained by zeroing the main
diagonal of $A^{\prime}$ and $1\leq k\leq n,$ then%
\[
\lambda_{k}\left(  A\right)  \geq\lambda_{k}(A^{\prime})-1.
\]

\end{proposition}

Indeed, $X=A^{\prime}-A$ is a $\left(  0,1\right)  $-diagonal matrix, and so
$\lambda_{1}\left(  A\right)  \leq1$. Therefore,
\[
\lambda_{k}\left(  A\right)  +1\geq\lambda_{k}\left(  A\right)  +\lambda
_{1}\left(  A\right)  \geq\lambda_{k}(A^{\prime}).
\]

\begin{proposition}
\label{proW}If $G$ is a graph of order $n$ and $2\leq k\leq n,$ then
\[
\lambda_{k}\left(  G\right)  +\lambda_{n-k+2}(\overline{G})\leq-1.
\]

\end{proposition}

Indeed, if $A$ and $\overline{A}$ are the adjacency matrices of $G$ and
$\overline{G},$ then $A+\overline{A}$ is the adjacency matrix of the complete
graph $K_{n}.$ Hence $\lambda_{k}\left(  G\right)  +\lambda_{n-k+2}%
(\overline{G})\leq\lambda_{k}\left(  K_{n}\right)  =-1.$\medskip

\subsection{\textbf{Blowups of graphs and their eigenvalues}}

Given a graph $G$ and an integer $t\geq1,$ replace each vertex of $G$ by an
independent set on $t$ vertices and each edge of $G$ by a complete bipartite
graph $K_{t,t}.$ Write $G^{\left(  t\right)  }$ for the resulting graph and
call it a \emph{blowup} of $G$.

If $G$\ is a graph of order $n,$ then $G^{\left(  t\right)  }$ is a graph of
order $nt$ and its adjacency matrix $A\left(  G^{\left(  t\right)  }\right)  $
is given by the equation
\[
A\left(  G^{\left(  t\right)  }\right)  =A\left(  G\right)  \otimes J_{t}.
\]

This algebraic representation of $A\left(  G^{\left(  t\right)  }\right)  $
gives a key to its spectrum:

\begin{proposition}
\label{prob}If $t\geq1$ and $G$ is a graph of order $n,$ with eigenvalues
$\lambda_{1}\left(  G\right)  ,\ldots,\lambda_{n}\left(  G\right)  ,$ then the
eigenvalues of $G^{\left(  t\right)  }$ are $\lambda_{1}\left(  G\right)
t,\ldots,\lambda_{n}\left(  G\right)  t,$ together with $\left(  t-1\right)
n$ additional zeros.
\end{proposition}

Most often we shall use the following variation of the blow-up operation:
given a graph $G$ and an integer $t\geq1,$ replace each vertex of $G$ by a
complete graph on $t$ vertices and each edge of $G$ by a complete bipartite
graph $K_{t,t}.$ Write $G^{\left[  t\right]  }$ for the resulting graph and
call it a \emph{closed blowup} of $G$.

If $G$\ is a graph of order $n,$ then $G^{\left[  t\right]  }$ is a graph of
order $nt$ and its adjacency matrix $A\left(  G^{\left[  t\right]  }\right)  $
is given by the equation
\[
A\left(  G^{\left[  t\right]  }\right)  =\left(  A\left(  G\right)
+I_{n}\right)  \otimes J_{t}-I_{nt}.
\]

This algebraic representation of $A\left(  G^{\left[  t\right]  }\right)  $
can be used to find the spectrum of $G^{\left[  t\right]  }$:

\begin{proposition}
\label{probu}If $t\geq1$ and $G$ is a graph of order $n,$ with eigenvalues
$\lambda_{1}\left(  G\right)  ,\ldots,\lambda_{n}\left(  G\right)  ,$ then the
eigenvalues of $G^{\left[  t\right]  }$ are $\lambda_{1}\left(  G\right)
t+t-1,\ldots,\lambda_{n}\left(  G\right)  t+t-1,$ together with $\left(
t-1\right)  n$ additional $-1$'s.
\end{proposition}

\medskip

\subsection{\label{TS}Taylor's strongly regular graphs and their complements}

In \cite{Tay71,Tay77} Taylor came up with a remarkable family of strongly
regular graphs $T\left(  q\right)  ,$ defined for every odd prime power $q,$
and with parameters%
\[
v=q^{3},\text{ \ }k=\frac{1}{2}\left(  q-1\right)  \left(  q^{2}+1\right)
,\text{ \ }a=\frac{1}{4}\left(  q-1\right)  ^{3}-1,\text{ \ \ }c=\frac{1}%
{4}\left(  q-1\right)  \left(  q^{2}+1\right)  ,
\]
Following the general rules, one finds that the eigenvalues of $T\left(
q\right)  $ are%
\begin{align*}
\lambda_{1}\left(  T\left(  q\right)  \right)   &  =\text{ \ }\frac{1}%
{2}\left(  q-1\right)  \left(  q^{2}+1\right)  \text{ with multiplicity 1};\\
\lambda_{2}\left(  T\left(  q\right)  \right)   &  =\text{ \ }\frac{1}%
{2}\left(  q-1\right)  ,\text{\ with multiplicity }\left(  q-1\right)  \left(
q^{2}+1\right)  ;\\
\lambda_{n}\left(  T\left(  q\right)  \right)   &  =-\frac{1}{2}\left(
q^{2}+1\right)  ,\text{ with multiplicity \ }q\left(  q-1\right)  .
\end{align*}
The complement $\overline{T\left(  q\right)  }$ is a strongly regular graph
with parameters%
\[
v=q^{3},\text{ \ }k=\frac{1}{2}\left(  q+1\right)  \left(  q^{2}-1\right)
,\text{ \ }a=\frac{1}{4}\left(  q+3\right)  \left(  q^{2}-3\right)  ,\text{
\ \ }c=\frac{1}{4}\left(  q+1\right)  \left(  q^{2}-1\right)  .
\]
For the eigenvalues of $\overline{T\left(  q\right)  }$ one finds that
\begin{align}
\lambda_{1}(\overline{T\left(  q\right)  }) &  =\text{ \ }\frac{1}{2}\left(
q+1\right)  \left(  q^{2}-1\right)  \text{ with multiplicity 1};\nonumber\\
\lambda_{2}(\overline{T\left(  q\right)  }) &  =\text{ \ }\frac{1}{2}\left(
q^{2}-1\right)  ,\text{ with multiplicity \ }q\left(  q-1\right)
;\label{ref}\\
\lambda_{n}(\overline{T\left(  q\right)  }) &  =-\frac{1}{2}\left(
q+1\right)  ,\text{\ with multiplicity }\left(  q-1\right)  \left(
q^{2}+1\right)  .\nonumber
\end{align}

\subsubsection{Some analytic properties of Taylor graphs}

Below we focus on certain properties of the Taylor graphs that may be of
interest to researchers in spectral and extremal graph theory, as well as in
quasi-random (\cite{CGW89}) and pseudo-random (\cite{KrSu06}, \cite{Tho1},
\cite{Tho2}) graphs. To simplify the view on the graph $T\left(  q\right)  $
for sufficiently large $q,$ we let $q^{3}=n,$ and disregard low order terms
when needed. Then $T\left(  q\right)  $ is a $\left(  n/2\right)  $-regular
graph $G$ of order $n,$ with the following properties:

\begin{enumerate}
\item Every two distinct vertices of $G$ have $\thickapprox n/4$ common
neighbors, and the same holds for $\overline{G}$. Therefore, both $G$ and
$\overline{G}$ are quasi-random (pseudo-random) graphs of density $1/2$;

\item For the spectrum of $G$ one finds that
\begin{align*}
\lambda_{1}\left(  G\right)   &  \thickapprox\text{ \ }n/2\text{ with
multiplicity }1;\\
\lambda_{2}\left(  G\right)   &  \thickapprox\text{ \ }n^{1/3}/2,\text{\ with
multiplicity}\thickapprox n;\\
\lambda_{n}\left(  G\right)   &  \thickapprox-n^{2/3}/2,\text{\ with
multiplicity\ }n^{2/3}.
\end{align*}
Therefore, almost all eigenvalues of $G$ are positive.

\item Nonetheless, the sum of squares of the non-principal positive
eigenvalues of $G$ is a vanishing proportion of the sum of squares of all
eigenvalues:
\[
\sum_{\lambda_{i}\left(  G\right)  >0,i>1}\lambda_{i}^{2}\left(  G\right)
\thickapprox\frac{1}{2}n^{5/3}=o\left(  1\right)  \sum_{i=1}^{n}\lambda
_{i}^{2}\left(  G\right)  =o\left(  1\right)  e\left(  G\right)  .
\]

\item For the spectrum of $\overline{G}$ one finds that
\begin{align*}
\lambda_{1}(\overline{G}) &  \thickapprox\text{ \ }n/2\text{ with multiplicity
}1;\\
\lambda_{2}(\overline{G}) &  \thickapprox\text{ \ }n^{2/3}/2,\text{ \ with
multiplicity}\thickapprox n^{2/3};\\
\lambda_{n}(\overline{G}) &  \thickapprox-n^{1/3}/2,\text{\ with
multiplicity}\thickapprox n.
\end{align*}
Therefore, almost all eigenvalues of $\overline{G}$ are negative.

\item Nonetheless, the sum of squares of the negative eigenvalues of
$\overline{G}$ is a vanishing proportion of the sum of squares of all
eigenvalues:%
\[
\sum_{\lambda_{i}(\overline{G})<0}\lambda_{i}^{2}(\overline{G})\thickapprox
\frac{1}{2}n^{5/3}=o\left(  1\right)  \sum_{i=1}^{n}\lambda_{i}^{2}%
(\overline{G})=o\left(  1\right)  e(\overline{G}).
\]

\end{enumerate}

\medskip

\subsection{\label{LS}Some symmetric Latin squares}

In the proofs of Theorems \ref{thKHN}, \ref{thj}, and \ref{thj1} we shall use
two types of symmetric Latin squares: back-circulant Latin square and
symmetric Latin square with constant diagonal. These constructions are simple
and well-known, but for reader's sake we shall describe them below.

Let $s$ be a positive integer. The back-circulant Latin square of size $s$
with symbol set $\left\{  1,\ldots,s\right\}  $ is an $s\times s$ square
matrix $L=\left[  l_{i,j}\right]  ,$ with $l_{i,j}$ given by
\[
l_{i,j}=(\left(  i+j\right)  \text{ }\operatorname{mod}\text{ }s)+1,\text{
}1\leq i,j\leq s.
\]
Obviously $L$ is a symmetric Latin square and its entries belong to $\left\{
1,\ldots,s\right\}  .$

Note that if $L$ is a symmetric Latin square of odd order with symbol set $S$,
then every symbol $s\in S$ occurs above the main diagonal as many times as
below it; hence, $s$ also occurs on the main diagonal, as the total number of
occurrences of $s$ is odd. Therefore, the main diagonal of $L$ contains each
symbol exactly once.

Next we want to construct symmetric Latin squares with constant diagonals. By
the above observation, the order of such Latin square cannot be odd, and for
any even $s,$ we shall give a construction, which seems well-known: we borrow
it from \cite{IoSh06}. Thus, let $s$ be an even positive integer, and define
an $s\times s$ square matrix $L=\left[  l_{i,j}\right]  ,$ with entries given
by
\[
l_{i,j}=\left\{
\begin{array}
[c]{ll}%
s, & \text{if \ }1\leq i\leq s\text{ \ and }i=j;\\
(\left(  i+j\right)  \operatorname{mod}\text{ }\left(  s-1\right)  )+1, &
\text{if \ }1\leq i<s,\text{ }1\leq j<s,\text{ and }i\neq j;\\
(\text{ }2j\text{ }\operatorname{mod}\text{ }\left(  s-1\right)  )+1, &
\text{if \ }i=s\text{ and\ \ }1\leq j<s;\\
(\text{ }2i\text{ }\operatorname{mod}\text{ }\left(  s-1\right)  )+1, &
\text{if \ }1\leq i<s\text{ \ and }j=s.
\end{array}
\right.
\]
The matrix $L$ is a symmetric Latin square with symbol set $\left\{
1,\ldots,s\right\}  $ and the symbol $s$ along the main diagonal. For example,
for $s=2,4,$ and $6,$ this construction gives%
\[
L=\left[
\begin{array}
[c]{cc}%
2 & 1\\
1 & 2
\end{array}
\right]  ,\text{ \ }L=\left[
\begin{array}
[c]{cccc}%
4 & 1 & 2 & 3\\
1 & 4 & 3 & 2\\
2 & 3 & 4 & 1\\
3 & 2 & 1 & 4
\end{array}
\right]  \text{ \ }L=\left[
\begin{array}
[c]{cccccc}%
6 & 4 & 5 & 1 & 2 & 3\\
4 & 6 & 1 & 2 & 3 & 5\\
5 & 1 & 6 & 3 & 5 & 2\\
1 & 2 & 3 & 6 & 2 & 4\\
2 & 3 & 5 & 2 & 6 & 1\\
3 & 5 & 2 & 4 & 1 & 6
\end{array}
\right]  .
\]
\medskip

\section{\label{ps}Proofs of some theorems}

\bigskip

\subsection{Proofs of Theorems \ref{th2}, \ref{thglb}, and \ref{thSB}}

\medskip

\begin{proof}
[\textbf{Proof of Theorem \ref{th2}}]We shall prove only the first bound, as
the other one follows by Proposition \ref{spro}. Let $\overline{T\left(
q\right)  }$ be the complement of the Taylor strongly regular graph $T\left(
q\right)  $ of order $q^{3}$ (see \ref{TS} for details). Let $G$ be a closed
blowup of $\overline{T\left(  q\right)  },$ i.e., $G=\overline{T\left(
q\right)  }^{\left[  t\right]  },$ and let $n=tq^{3}=v\left(  G\right)  .$
Then
\[
c_{k}\geq\sup\frac{\lambda_{k}\left(  G\right)  }{n}=\sup\frac{\lambda
_{q\left(  q-1\right)  +1}(\overline{T\left(  q\right)  })t+t-1}{q^{3}t}%
=\sup\frac{q^{2}+1}{2q^{3}}-\frac{1}{q^{3}t}=\frac{q^{2}+1}{2q^{3}}.
\]
To finish the proof we need to show that
\[
\frac{q^{2}+1}{2q^{3}}\geq\frac{1}{2\sqrt{q^{2}-q}+1}=\frac{1}{2\sqrt{k-1}%
+1}.
\]
This inequality follows from%
\[
\frac{q^{2}+1}{2q^{3}}>\frac{2q-1}{4q^{2}+2q-1}>\frac{1}{2\sqrt{q^{2}-q}+1}%
\]
after some simple algebra, which we omit.
\end{proof}

\bigskip

\begin{proof}
[\textbf{Proof of Theorem \ref{thglb}}]Fix a sufficiently large integer $k,$
and let $q$ be the smallest prime such that%
\[
q\left(  q-1\right)  +1\geq k.
\]
A result of Baker, Harman, and Pintz \cite{BHP01} on the distribution of
primes implies that if $k$ is sufficiently large, then%
\[
q\leq\sqrt{k}+1/2+\left(  \sqrt{k}+1/2\right)  ^{21/40}.
\]
It is not hard to see that if $k$ is sufficiently large, then $q<\sqrt
{k-1}+\sqrt[3]{k}/2.$ Let $\overline{T\left(  q\right)  }$ be the complement
of the Taylor graph $T\left(  q\right)  ,$ let $G$ be a closed blowup of
$\overline{T\left(  q\right)  }$, i.e., $G=$ $\overline{T\left(  q\right)
}^{\left[  t\right]  },$ and set $n=tq^{3}=v\left(  G\right)  .$ We see that
\[
\lambda_{k}\left(  G\right)  =\lambda_{2}(\overline{T\left(  q\right)
})t+t-1=\frac{1}{2}\left(  q^{2}-1\right)  t+t-1>\frac{1}{2}q^{2}t.
\]
Hence,
\[
c_{k}\geq\frac{\lambda_{k}\left(  G\right)  }{n}>\frac{q^{2}t}{2q^{3}t}%
=\frac{1}{2q}\geq\frac{1}{2\sqrt{k-1}+\sqrt[3]{k}}.
\]
Now, the bound on $c_{-k+1}$ follows in view of Proposition \ref{spro}.
\end{proof}

\bigskip

\begin{proof}
[\textbf{Proof of Theorem \ref{thSB}}]In view of Proposition \ref{spro}, if
$k\geq2$, we always have $c_{-k+1}\geq c_{k},$ so our main goal is to prove
the bounds on $c_{k}.$ This proof naturally splits into two cases: $5\leq
k\leq15$ and $k\geq16.$

If $5\leq k\leq15,$ we just take an appropriate strongly regular graph $H$
with parameters $(v,k,a,c),$ and let $G$ be a closed blowup of $H\ $of order
$n,$ i.e., $G=H^{\left[  n/v\right]  },$ where $n$ is a multiple of $v.$ In
view of Proposition \ref{probu},
\[
\lambda_{k}\left(  G\right)  =\left(  \lambda_{k}\left(  H\right)  +1\right)
\frac{n}{v}-1.
\]
Using well known sources, say the home page of A. Brouwer, we obtain the
following table:%
\[%
\begin{tabular}
[t]{ccc}%
\ $H(v,k,a,c)$ & \ Eigenvalue of $H$ \  & Eigenvalue of $G=H^{\left[
n/v\right]  }\medskip$\\
(9, 4, 1, 2) & $\lambda_{5}\left(  H\right)  =1$ & $\lambda_{5}\left(
G\right)  =\frac{2}{9}n-1\medskip$\\
(10, 3, 0, 1) & $\lambda_{6}\left(  H\right)  =1$ & $\lambda_{6}\left(
G\right)  =\frac{1}{5}n-1\medskip$\\
(13, 6, 2, 3) & \ \ \ \ \ \ $\lambda_{7}\left(  H\right)  =\frac{\sqrt{13}%
-1}{2}$ & \ \ \ \ \ $\lambda_{7}\left(  G\right)  =\frac{\sqrt{13}+1}%
{26}n-1\medskip$\\
(15, 6, 1, 3) & $\lambda_{8}\left(  H\right)  =1$ & $\lambda_{8}\left(
G\right)  =\frac{2}{15}n-1\medskip$\\
(15, 6, 1, 3) & $\lambda_{9}\left(  H\right)  =1$ & $\lambda_{9}\left(
G\right)  =\frac{2}{15}n-1\medskip$\\
(15, 6, 1, 3) & $\lambda_{10}\left(  H\right)  =1$ & $\lambda_{10}\left(
G\right)  =\frac{2}{15}n-1\medskip$\\
(21,10,3,6) & $\lambda_{11}\left(  H\right)  =1$ & $\lambda_{11}\left(
G\right)  =\frac{2}{21}n-1\medskip$\\
(21,10,3,6) & $\lambda_{12}\left(  H\right)  =1$ & $\lambda_{12}\left(
G\right)  =\frac{2}{21}n-1\medskip$\\
(21,10,3,6) & $\lambda_{13}\left(  H\right)  =1$ & $\lambda_{13}\left(
G\right)  =\frac{2}{21}n-1\medskip$\\
(21,10,3,6) & $\lambda_{14}\left(  H\right)  =1$ & $\lambda_{14}\left(
G\right)  =\frac{2}{21}n-1\medskip$\\
(21,10,3,6) & $\lambda_{15}\left(  H\right)  =1$ & $\lambda_{15}\left(
G\right)  =\frac{2}{21}n-1\medskip$%
\end{tabular}
\
\]
Now, letting $n\rightarrow\infty,$ we obtain
\[
c_{5}\geq2/9,\text{ \ }c_{6}\geq1/5,\text{ \ }c_{7}\geq\sqrt{13}/2+1/2,\text{
\ }c_{8}\geq2/15,\text{ \ }c_{9}\geq2/15,\text{ \ }c_{10}\geq2/15.
\]
Likewise, if $11\leq k\leq15,$ we obtain $c_{k}\geq2/21.$ These inequalities
obviously imply that
\[
c_{k}\geq\frac{1}{k-1/2}%
\]
whenever $5\leq k\leq15.$

Now, let $k\geq16,$ and let $q$ be the smallest prime $q$ such that
\begin{equation}
q\geq1/2+\sqrt{k-3/4}.\label{ink}%
\end{equation}
Bertrand's postulate guarantees that for any real $x>3,$ there is a prime $q$
such that%
\[
\left\lceil x\right\rceil <q\leq2\left\lceil x\right\rceil -3.
\]
Since $2\left\lceil x\right\rceil -3<2x-1,$ in our case this implies that
\begin{equation}
q<2\left(  1/2+\sqrt{k-1}\right)  -1=\sqrt{4k-3}<2\sqrt{k-1}.\label{ink1}%
\end{equation}
Let $\overline{T\left(  q\right)  }$ be the complement of the Taylor graph
$T\left(  q\right)  ,$ and let $G=$ $\overline{T\left(  q\right)  }^{\left[
t\right]  }.$ Since inequality (\ref{ink}) implies that $k\leq q\left(
q-1\right)  +1,$ in view of (\ref{ref}), we see that%
\[
\lambda_{k}(\overline{T\left(  q\right)  })=\frac{1}{2}\left(  q^{2}-1\right)
,
\]
and therefore,%
\[
\lambda_{k}\left(  G\right)  =\lambda_{k}(\overline{T\left(  q\right)
})t+t-1=\frac{1}{2}\left(  q^{2}-1\right)  t+t-1>\frac{1}{2}q^{2}t.
\]
Now, inequality (\ref{ink1}) implies that
\[
c_{k}^{\ast}\geq\frac{\lambda_{k}\left(  G\right)  }{q^{3}t}>\frac{1}%
{2q}>\frac{1}{4\sqrt{k-1}},
\]
completing the proof of Theorem \ref{thSB}.
\end{proof}

\bigskip

\subsection{Proofs of Theorems \ref{thKHN}, \ref{thj}, and \ref{thj1}}

The proofs of Theorems \ref{thKHN}, \ref{thj}, and \ref{thj1} are very close,
but for reader's sake we give them separately. All three proofs exploit the
construction of Hadamard matrices due to Kharaghani \cite{Kha85}, see also
\cite{IoSh06}, Theorem 4.4.16. We shall vary both the blocks and the
underlying Latin square, so the reader is referred to \ref{LS} for necessary
details about Latin squares. The idea of using symmetric Latin squares with
constant diagonal is borrowed from Haemers \cite{Hae08}, Theorem 2, and Ionin
and Shrikhande \cite{IoSh06}, Corollary 5.3.17.\medskip

\begin{proof}
[\textbf{Proof of Theorem \ref{thKHN}}]Suppose that $L=\left[  l_{i,j}\right]
$ is a back-circulant Latin square of size $s,$ with symbol set $\left\{
1,\ldots,s\right\}  .$ Let $\mathbf{x}_{1},\ldots,\mathbf{x}_{s}$ be
orthogonal $\left(  -1,1\right)  $-vectors of dimension $n\geq s$ that are
also orthogonal to the all ones vector $\mathbf{j}_{n}.$ An easy choice is to
take the last $k$ rows of a normalized Hadamard matrix of order $n>k.$ For
each $s=1,\ldots,s,$ define a square matrix $A_{s}$ by $A_{s}=-\mathbf{x}%
_{s}\otimes\mathbf{x}_{s}.$ Obviously $A_{1},\ldots,A_{s}$ are symmetric
$\left(  -1,1\right)  $-matrices of size $n$ and rank $1,$ with diagonal
entries equal to $-1.$

Now, let $B$ be the block matrix obtained by replacing each entry $l_{i,j}$ of
$L$ by the matrix $A_{l_{i,j}}.$ Note that $B$ is a symmetric $\left(
-1,1\right)  $-matrix of size $sn.$ Obviously the diagonal entries of $B$ are
equal to $-1,$ thus \emph{(iii) }holds. Note also that
\[
A_{p}A_{q}=\left\{
\begin{array}
[c]{ll}%
0 & \text{if }q\neq p;\\
-nA_{p} & \text{if }q=p.
\end{array}
\right.
\]
Hence, $B^{2}$ is block diagonal with each diagonal block equal to
$-nA_{1}-\cdots-nA_{s}.$ But $-nA_{1}-\cdots-nA_{s}$ is of rank $s,$ so it has
exactly $s$ nonzero eigenvalues, each equal to $n^{2}.$ Thus, $B$ has $s^{2}$
nonzero eigenvalues, and their absolute value is equal to $n.$

Writing $n_{+}$ and $n_{-}$ for the number of positive and negative
eigenvalues of $B,$ we have%
\[
\left(  n_{+}-n_{-}\right)  n=\mathrm{tr}\text{ }B=-sn,
\]
and so, $n_{+}-n_{-}=-s,$ implying that
\[
n_{+}=\binom{s-1}{2}\text{ and }n_{-}=\binom{s+1}{2},
\]
completing the proof of \emph{(i)}.

To prove \emph{(ii)} note that each rowsum of each matrix $A_{i}$ is zero, so
the rowsums of $B$ are zero as well. This completes the proof of Theorem
\ref{thKHN}.
\end{proof}

\bigskip

\begin{proof}
[\textbf{Proof of Theorem \ref{thj}}]Suppose that $L=\left[  l_{i,j}\right]  $
is a symmetric Latin square of order $s$ with constant diagonal. Let $\left\{
1,\ldots,s\right\}  $ be the symbol set of $L$ and $s$ be the diagonal symbol.
Next, select $s$ vectors $\mathbf{x}_{1},\ldots,\mathbf{x}_{s}$ of dimension
$n\geq s$ such that $\mathbf{x}_{1}=\mathbf{j}_{n}$ and every two of the
vectors $\mathbf{x}_{1},\ldots,\mathbf{x}_{s}$ are orthogonal$.$ An easy
choice is to take the first $k$ rows of a normalized Hadamard matrix of order
$n\geq k.$ Let $A_{1}=-J_{n},$ and for each $i=2,\ldots,s,$ define a square
matrix $A_{i}$ by
\[
A_{i}=\mathbf{x}_{i}\otimes\mathbf{x}_{i}.
\]
Obviously $A_{1},\ldots,A_{s}$ are symmetric $\left(  -1,1\right)  $-matrices
of size $n$ and rank $1.$ Note also that the diagonal entries of $A_{s}$ are
equal to $1.$

Now, let $B$ be the block matrix obtained by replacing each entry $l_{i,j}$ of
$L$ by the matrix $A_{l_{i,j}}.$ Note that $B$ is a symmetric $\left(
-1,1\right)  $-matrix of size $sn.$ The diagonal entries of $B$ are equal to
$1,$ thus \emph{(iii) }holds. Note also that%
\[
A_{p}A_{q}=\left\{
\begin{array}
[c]{ll}%
0 & \text{if }q\neq p;\\
nJ_{n} & \text{if }q=p=1;\\
nA_{p} & \text{if }q=p\neq1.
\end{array}
\right.
\]
Hence, $B^{2}$ is block diagonal with each diagonal block equal to
$nJ_{n}+\cdots+nA_{s}.$ Since $nJ_{n}+\cdots+nA_{s}$ has exactly $s$ nonzero
eigenvalues, each equal to $n^{2},$ we see that $B$ has $s^{2}$ nonzero
eigenvalues, and their absolute value is equal to $n.$

Writing $n_{+}$ and $n_{-}$ for the number of positive and negative
eigenvalues of $B,$ we have%
\[
\left(  n_{+}-n_{-}\right)  n=\mathrm{tr}\text{ }B=sn,
\]
and so, $n_{+}-n_{-}=s,$ implying that
\[
n_{+}=\binom{s+1}{2}\text{ and }n_{-}=\binom{s-1}{2},
\]
completing the proof of \emph{(i)}.

To prove \emph{(ii)} note that\emph{ }%
\[
A_{p}\mathbf{j}_{n}=\left\{
\begin{array}
[c]{cc}%
0 & \text{if }p\neq1;\\
-n\mathbf{j}_{n} & \text{if }p=1.
\end{array}
\right.
\]
Therefore,
\[
B\mathbf{j}_{sn}=\left(
\begin{array}
[c]{c}%
A_{1}\mathbf{j}_{n}\\
A_{1}\mathbf{j}_{n}\\
\vdots\\
A_{1}\mathbf{j}_{n}%
\end{array}
\right)  =-n\mathbf{j}_{sn}.
\]
Thus, $\mathbf{j}_{kn}$ is an eigenvector of $B$ to the eigenvalue $-n,$
completing the proof of Theorem \ref{thj}.
\end{proof}

\bigskip

\begin{proof}
[\textbf{Proof of Theorem \ref{thj1}}]Our proof combines the proofs of
Theorems \ref{thKHN} and \ref{thj}. Suppose that $L=\left[  l_{i,j}\right]  $
is a back-circulant Latin square of size $s,$ with symbol set $\left\{
1,\ldots,s\right\}  .$ Next, select $s$ vectors $\mathbf{x}_{1},\ldots
,\mathbf{x}_{s}$ of dimension $n\geq s$ such that $\mathbf{x}_{1}%
=\mathbf{j}_{n}$ and every two of the vectors $\mathbf{x}_{1},\ldots
,\mathbf{x}_{s}$ are orthogonal$.$ Let $A_{1}=-J_{n},$ and for each
$i=2,\ldots,s,$ define a square matrix $A_{i}$ by $A_{i}=\mathbf{x}_{i}%
\otimes\mathbf{x}_{i}.$ Let $B$ be the block matrix obtained by replacing each
entry $l_{i,j}$ of $L$ by the matrix $A_{l_{i,j}}.$ Note that $B$ is a
symmetric $\left(  -1,1\right)  $-matrix of size $sn,$ with $\left(
s-1\right)  n$ diagonal entries equal to $1$ and $n$ diagonal entries equal to
$-1.$ Note also that%
\[
A_{p}A_{q}=\left\{
\begin{array}
[c]{ll}%
0 & \text{if }q\neq p;\\
nJ_{n} & \text{if }q=p=1;\\
nA_{p} & \text{if }q=p\neq1.
\end{array}
\right.
\]
Hence, $B$ has $s^{2}$ nonzero eigenvalues, and their absolute value is equal
to $n.$ Writing $n_{+}$ and $n_{-}$ for the number of positive and negative
eigenvalues of $B,$ we have%
\[
\left(  n_{+}-n_{-}\right)  n=\mathrm{tr}\text{ }B=\left(  s-2\right)  n,
\]
and so, $n_{+}-n_{-}=s-2,$ implying that
\[
n_{+}=\binom{s+1}{2}-1\text{ and }n_{-}=\binom{s-1}{2}+1,
\]
completing the proof of \emph{(i)}.

To prove \emph{(ii)} let us note that if $2\leq i\leq s$, then all rowsums of
$A_{p}$ are zero. So the rowsums of $B$ are equal to the rowsums of $-J_{n},$
which are equal to $-n.$
\end{proof}

\medskip

\subsection{Proofs of Theorems \ref{thck} and \ref{thck1} \emph{ }}

\medskip

\begin{proof}
[\textbf{Proof of Theorem \ref{thck}}]Let $B$ be a matrix constructed by
Theorem \ref{thj} and let $C=B\otimes J_{t}.$ The properties of $B$ given by
Theorem \ref{thj} imply that $C$ is a symmetric $\left(  -1,1\right)  $-matrix
of order $snt$ such that:

- $C$ has exactly $s^{2}$ nonzero eigenvalues, of which $\binom{s-1}{2}$ are
equal to $nt$ and $\binom{s+1}{2}$ are equal to $-nt$;

- the diagonal entries of $C$ are equal to $-1$;

- the vector $\mathbf{j}_{snt}$ is an eigenvector of $C$ to the eigenvalue
$-nt.\medskip$

Now, let
\[
A=\frac{1}{2}\left(  J_{snt}-C\right)  .
\]
Clearly $A$ is a symmetric $\left(  0,1\right)  $-matrix of order $snt$, with
zero diagonal; hence, $A$ is the adjacency matrix of some graph $G$ of order
$snt$. Let $\mathbf{x}_{1}=\mathbf{j}_{snt},\mathbf{x}_{2},\ldots
,\mathbf{x}_{snt}$ be orthogonal eigenvectors to $C.$ Note that
\[
\frac{1}{2}\left(  J_{snt}-C\right)  \mathbf{j}_{snt}=\left(  \frac{snt}%
{2}+\frac{nt}{2}\right)  \mathbf{j}_{snt},
\]
so $snt/2+nt/2$ is an eigenvalue to $G.$ Also, for any $i=2,\ldots,snt,$ we
see that
\[
A\mathbf{x}_{i}=\frac{1}{2}\left(  J_{snt}-C\right)  \mathbf{x}_{i}=-\frac
{1}{2}C\mathbf{x}_{i},
\]
so $G$ has $\binom{s+1}{2}$ eigenvalues equal to $-nt/2$ and $\binom{s-1}%
{2}-1$ eigenvalues equal to $nt/2.$ Therefore,%
\[
\lambda_{1}^{\ast}\left(  G\right)  +\cdots+\lambda_{s^{2}}^{\ast}\left(
G\right)  =\frac{snt}{2}+\frac{nt}{2}+\left(  s^{2}-1\right)  \frac{nt}%
{2}=\frac{1}{2}\left(  1+s\right)  \frac{snt}{2},
\]
completing the proof of Theorem \ref{thck}.
\end{proof}

\medskip

\begin{proof}
[\textbf{Proof of Theorem \ref{thck1}}]Let $\lambda$ be the rowsum of $B,$
which clearly is a nonzero eigenvalue of $B$ with eigenvector $\mathbf{j}_{n}%
$. Since $B\in\mathbb{S}_{k},$ either $\lambda=n/\sqrt{k}$ or $\lambda
=-n/\sqrt{k}.$ We shall assume that $\lambda=-n/\sqrt{k},$ for otherwise we
just take $-B$ for $B.$ Now, for any positive integer $t,$ define a symmetric
$\left(  0,1\right)  $-matrix $A^{\prime}$ by
\[
A^{\prime}=\frac{1}{2}\left(  J_{nt}-B\otimes J_{t}\right)  ,
\]
and note that%
\[
\lambda_{1}\left(  A^{\prime}\right)  =\frac{nt}{2}+\frac{nt}{2\sqrt{k}%
},\text{ \ \ and \ \ }\lambda_{i}^{\ast}\left(  A^{\prime}\right)  =\frac
{nt}{2\sqrt{k}}\text{ for }1<i\leq k.
\]
Next, zero the diagonal of $A^{\prime}$ and write $A$ for the resulting
matrix. Clearly $A$ is a symmetric $\left(  0,1\right)  $-matrix with zero
diagonal, so $A$ is the adjacency matrix of some graph $G\ $of order $nt.$
Using Weyl's inequalities (Proposition \ref{proW0}), we see that
\[
\lambda_{1}\left(  G\right)  \geq\frac{nt}{2}+\frac{nt}{2\sqrt{k}}-1,\text{
\ \ and \ \ }\lambda_{i}^{\ast}\left(  G\right)  \geq\frac{nt}{2\sqrt{k}%
}-1\text{ for }1<i\leq k.
\]
Therefore,
\[
\lambda_{1}^{\ast}\left(  G\right)  +\cdots+\lambda_{k}^{\ast}\left(
G\right)  \geq\left(  \frac{nt}{2}+k\frac{nt}{2\sqrt{k}}\right)  -k=\frac
{1}{2}\left(  1+\sqrt{k}\right)  nt-k,
\]
completing the proof of Theorem \ref{thck1}.
\end{proof}

\medskip

\section{\label{RM}A recap for symmetric $\left(  -1,1\right)  $-matrices}

Many solutions in this paper come from $\left(  -1,1\right)  $-matrices. This
is not incidental, for if $G$ is a regular graph, then its adjacency spectrum
is linearly equivalent to the spectrum of its Seidel's matrix, which is a
$\left(  0,-1,1\right)  $-matrix. But there is more to that: if $G$ is a
$\left(  n/2\right)  $-regular graph, the Seidel matrix effectively eliminates
the largest eigenvalue of $G$, which may be nuisance in certain spectral
problems, like most of the problems discussed in this paper. One cannot but
agree that many of the questions raised above for graphs seem more balanced
and natural if translated for $\left(  -1,1\right)  $-matrices. In this
section we explore such translations.

Thus, for any $k\geq1,$ let us introduce the functions%
\[
\Lambda_{k}\left(  n\right)  =\max_{A\in\mathbb{U}_{n}}\lambda_{k}\left(
A\right)  \text{ \ \ and \ \ \ \ }\Lambda_{k}^{\ast}\left(  n\right)
=\max_{A\in\mathbb{U}_{n}}\lambda_{k}^{\ast}\left(  A\right)  .
\]
Obviously $\Lambda_{k}\left(  n\right)  $ and $\Lambda_{k}^{\ast}\left(
n\right)  $ are the matrix analogs of $\lambda_{k}\left(  n\right)  $ and
$\lambda_{k}^{\ast}\left(  n\right)  ;$ we do not need an analog to
$\lambda_{-k}\left(  n\right)  ,$ as $\mathbb{U}_{n}$ is closed under
negation. Next, in the general spirit of the paper, we raise the problem:

\begin{problem}
For any $k\geq2,$ find $\Lambda_{k}\left(  n\right)  $ and $\Lambda_{k}^{\ast
}\left(  n\right)  $.
\end{problem}

Much of what we have achieved for graphs applies to symmetric $\left(
-1,1\right)  $-matrices as well. First, obviously%
\begin{equation}
\Lambda_{k}\left(  n\right)  \leq\Lambda_{k}^{\ast}\left(  n\right)  \leq
n/\sqrt{k}.\label{umc}%
\end{equation}
Note that for $\Lambda_{k}^{\ast}\left(  n\right)  ,$ bound (\ref{umc}) is
precise for infinitely many $k$ and $n.$ Indeed, if \ $\mathbb{S}_{k}%
\neq\varnothing,$ for arbitrary large $n,$\ we have $\Lambda_{k}^{\ast}\left(
n\right)  =n/\sqrt{k}$.

Further, in analogy to $c_{k}$ and $c_{k}^{\ast},$ let
\[
d_{k}=\sup_{n\geq1}\frac{\Lambda_{k}\left(  n\right)  }{n}\text{ \ \ \ and
\ \ \ }d_{k}^{\ast}=\sup_{n\geq1}\frac{\Lambda_{k}^{\ast}\left(  n\right)
}{n}.
\]
The constants $d_{k}$ and $d_{k}^{\ast}$ are handy, as for any $n$ and any
matrix $A\in\mathbb{U}_{n},$ we have
\[
\lambda_{k}\left(  A\right)  \leq d_{k}n\text{ \ \ and \ \ }\lambda_{k}^{\ast
}\left(  A\right)  \leq d_{k}^{\ast}n.
\]
In turns out that these inequalities are best possible, for one can show that%
\[
\lim_{n\rightarrow\infty}\frac{\Lambda_{k}\left(  n\right)  }{n}=d_{k}\text{
\ \ \ \ and \ \ \ \ }\lim_{n\rightarrow\infty}\frac{\Lambda_{k}^{\ast}\left(
n\right)  }{n}=d_{k}^{\ast}.
\]
Thus, much about $\Lambda_{k}\left(  n\right)  $ and $\Lambda_{k}^{\ast
}\left(  n\right)  $ would be known if we knew $d_{k}$ and $d_{k}^{\ast}$ or
estimates thereof$.$ First, from (\ref{umc}) we immediately get an upper bound%
\[
d_{k}\leq d_{k}^{\ast}\leq1/\sqrt{k},
\]
so the difficulty is to find matching lower bounds.

As one may expect, $d_{k}^{\ast}$ is easier to tackle than $d_{k}.$ Indeed,
since $d_{k}^{\ast}$ is nonincreasing in $k,$ letting $s$ to be the smallest
positive integer such that $s^{2}\geq k,$ in view of $\mathbb{S}_{s^{2}}%
\neq\varnothing,$ we see that $d_{k}^{\ast}\geq1/s.$ But $\left(  s-1\right)
^{2}<k,$ and so,%
\[
\frac{1}{\sqrt{k}}\geq d_{k}^{\ast}\geq\frac{1}{s}>\frac{1}{\sqrt{k}+1}%
=\frac{1}{\sqrt{k}}+O\left(  k^{-1}\right)  .
\]
This argument does not fit to bound $d_{k}$, so we need another idea. Since
Taylor's graphs have been useful for $c_{k}$, we can hope to use them for
$d_{k}$ as well. Thus, let $A(\overline{T\left(  q\right)  })$ be the
adjacency matrix of the complement of the Taylor graph $T\left(  q\right)  $
of order $q^{3}.$ Define the matrix $T\in\mathbb{U}_{q^{3}}$, by setting
\[
T=2A(\overline{T\left(  q\right)  })-J_{q^{3}}.
\]
It is not hard to see that the matrix $T$ has three distinct eigenvalues:%
\begin{align*}
\lambda_{1}(T) &  =q^{2}-1,\text{ \ \ \ \ \ with multiplicity \ }q\left(
q-1\right)  ;\\
\lambda_{2}(T) &  =q^{2}-q-1,\text{ with multiplicity }1;\\
\lambda_{3}(T) &  =-q-1,\text{ \ \ \ \ with multiplicity }\left(  q-1\right)
\left(  q^{2}+1\right)  .
\end{align*}
Note in passing that the mapping $A(\overline{T\left(  q\right)  }\rightarrow
T$ preserves all eigenvalues except $\lambda_{1}(\overline{T\left(  q\right)
}),$ whose magnitude is reduced essentially to $\lambda_{2}(\overline{T\left(
q\right)  }).$

Now, using the Baker, Harman, and Pintz result \cite{BHP01} again, for
sufficiently large $k,$ we get the asymptotics
\[
\frac{1}{\sqrt{k}}\geq d_{k}\geq\frac{1}{\sqrt{k}+\sqrt[3]{k}}=\frac{1}%
{\sqrt{k}}+O\left(  k^{-2/3}\right)  .
\]
\medskip We end up with a question about the maximum Ky Fan $k$-norm of
matrices in $\mathbb{U}_{n}.$

\begin{problem}
For any $k\geq2,$ find $\max\limits_{A\in\mathbb{U}_{n}}\left\Vert
A\right\Vert _{\ast k}.$
\end{problem}

Without a proof, let us mention the bounds%
\[
(\sqrt{k}-1)n\leq\max_{A\in\mathbb{U}_{n}}\left\Vert A\right\Vert _{\ast
k}\leq n\sqrt{k}.
\]
\bigskip

\textbf{Acknowledgement. }Part of this paper has been prepared for a talk at
the Algebraic Combinatorics Workshop held in the Fall of 2014, at the
University of Science and Technology of China, Hefei. I am grateful for the
hospitality of the organizers, in particular to prof. Jack Koolen.

\end{document}